\documentclass[12pt]{amsart}

\usepackage{TDefn}

\DeclareMathOperator{\Obs}{Obs}
\renewcommand{\div}{\mathrm{div}}
\DeclareMathOperator{\Disc}{Disc}
\DeclareMathOperator{\Diff}{Diff}

\DeclareMathOperator{\prong}{prong}

\newcommand{\hol}{\mathrm{hol}}
\newcommand{\antihol}{\overline{\mathrm{hol}}}

\renewcommand{\top}{\mathrm{top}}

\newcommand{\mainbound}{m_1(\mu)/2 -  2g/3   }
\newcommand{\alternatemainbound}{\frac{m_1(\mu)}{6}  -  \sum_{j > 1}  \frac{jm_j(\mu)}{3}  - 2/3 }

\newcommand{\uGamma}{\underline{\Gamma}}

\newcommand{\geometricBound}{g - 3d/2 + 1}

\title{Stable homology of strata of abelian differentials}

\author{Philip Tosteson}

\begin{document}

\begin{abstract} We show that the homology of strata of abelian differentials stabilizes in a range where the number of simple zeros is large relative to the homological degree.   In this range, we show that the rational cohomology agrees with the restriction of the tautological classes to the stratum, and that the rational Picard group is trivial for unprojectivized strata.  Our proof method is to develop an $h$-principle for these strata, valid in a range of homological degrees that increases with the number of simple zeros.  
The same approach also applies to higher order differentials.
 \end{abstract}

\maketitle


\section{Introduction}


The Hodge bundle over the moduli space of genus $g$  curves $\cH_g = \{C \in \cM_g, \omega \in \Gamma(C, \Omega^1_C)\}$ admits a natural stratification into loci according to the multiplicities of the zeros of $\omega$.   These strata of abelian differentials play an important role in intersection theory of $\overline \cM_{g}$ and in Teichm\"uller dynamics.   Their connected components  were classified in a landmark paper of Kontsevich--Zorich \cite{KontsevichZorich}, but little has been known about their higher homology groups.    

 The homology of moduli spaces of curves stabilizes in the genus, as shown by Harer \cite{Harer}, and the stable rational homology was determined by Madsen--Weiss \cite{MadsenWeiss} by identifying the stable homology with the homology of an infinite loop space.  The purpose of this paper is to develop analogous results for strata of abelian differentials.

\subsection*{The stable rational cohomology} Let $g \in \bbN$ and  let $\mu = (\mu(1), \dots, \mu(n))$  be an integer partition of $2g-2$, so that $\sum_{i = 1}^n \mu(i) = 2g-2$.  The \textbf{stratum of abelian differentials} associated to this data is the space $$\cH_{g,\mu} := \{ C \in \cM_g, \omega \in H^0(\Omega^1_C) - 0 ~|~ \exists  p_1, \dots, p_n \in C \text{ distinct, }  {\rm div}(\omega) = \sum_{i} \mu(i) p_i\}.$$    Our first result computes the rational cohomology of $\cH_{g, \mu}$ in a range where the number of simple zeros is large.  To state it, let $m_j(\mu) := \#\{i ~|~ \mu(i) = j\}$ be the multiplicity of $j$ in $\mu$.    

\begin{thm}
	For $0< i \leq \mainbound$, the rational cohomology $H^i(\cH_{g,\mu}, \bbQ)$ vanishes. \label{mainComputation}
\end{thm}

	\begin{remark}\label{CL}
In contemporaneous work, Theorem \ref{mainComputation} has been obtained by Chen--Larson \cite{ChenLarson}, using a different method.  (See \S \ref{subsec:relationtootherwork} for further discussion).
\end{remark}

\begin{remark}
 	Since $2g-2 = \sum_{j} j m_j(\mu)$ we may restate the bound $i \leq \mainbound$ as $i \leq \alternatemainbound$, so that once  the vanishing profile of zeros of multiplicity $>1$ is fixed the range in which Theorem \ref{mainComputation} applies grows linearly  in the number of simple zeros.   We refer to this range as the \textbf{stable range}.  \end{remark}

The fact that the stable cohomology vanishes is consistent with the philosophy that the stable cohomology should coincide with a tautological ring, and  the fact that the tautological classes on $\cM_{g,n}$ restrict to zero on $\cH_{\mu}$ \cite{chen2019tautological}.  As a direct consequence, we obtain the homology of the projectivized stratum.    

\begin{cor} In the stable range, the cohomology of $H^*(\bbP \cH_{g,\mu}, \bbQ)$ agrees with a polynomial ring on $\eta = c_1(\cO(1))$.
\end{cor}

 Theorem \ref{mainComputation}  implies that the Borel--Moore homology spectral sequence for the stratification of $\cH_g$ by strata degenerates in the stable range, and so we obtain. 

\begin{cor} In the stable range, the fundamental classes $[\overline{\cH_{g,\mu}}] \in H^{\sum_{j} (\mu(j)-1)}(\cH_g)$ form a basis for $H^*(\cH_g, \bbQ) = H^*(\cM_g, \bbQ)$.
\end{cor}

Adapting an argument of Landesman--Levy \cite{LandLev}, we also have that the rational Picard group of $\cH_{g, \mu}$ vanishes in the stable range.  

\begin{cor}
	For $m_1(\mu) \geq 4 g/3 + 6$ we have $\Pic(\cH_{g,\mu}) \otimes \bbQ = 0$.
\end{cor}

 \subsection*{The stable integral cohomology}   In contrast to the rational (co)homology, the stable integral (co)homology is highly nontrivial.   We reduce the computation of the stable integral cohomology $H_*(\cH_{g,\mu},\bbZ)$ to a homotopical problem: computing the homology of the homotopy-quotient of a certain infinite loop space by a product of cyclic groups.   For example we have the following consequence.

\begin{thm} \label{modPExample} Let $p$ be a prime $> 2$.  In the stable range, the homology of $\cH_{g, 1^{2g-2}}$ with $\bbF_p$ coefficients is isomorphic to $$H_*((\Omega^2 (\Omega^\infty S^\infty), \bbF_p) \otimes H_*((\Omega^\infty S^\infty)_0, \bbF_p),$$ where $\Omega^n$ denotes the $n$th loop space functor and $\Omega^\infty S^\infty = \colim_{n} \Omega^n S^n$.    
The same holds with $1^{2g-2}$ replaced by any partition  $\mu$ such that $m_i(\mu) \leq 1$ for all $i > 0$ and $m_i(\mu)$ is only nonzero when $i+1$ is relatively prime to $p$.  
\end{thm}
Here the second tensor factor $H_*((\Omega^\infty S^\infty)_0, \bbF_p)$  is the stable homology of the symmetric group, see \cite[\S I.5]{May} for a full computation.  

 \begin{remark}
	Throughout, when we take (co)homology with integer coefficients, we consider $\cH_{g,\mu}$ as a stack.  Concretely, because the moduli space of curves can be presented as a global quotient stack, we may pass to a $G$-cover $X \to \cH_{g,\mu}$ for which $X$ is a scheme, and then we have $H^i(\cH_{g,\mu},\bbZ)  \iso  \bbH^i(G, C^*(X,\bbZ))$ and $H_i(\cH_{g,\mu},\bbZ)  \iso  \bbH_i(G, C_*(X,\bbZ))$, where $\bbH_i$ and $\bbH^i$ respectively denote hyperhomology and hypercohomology with coefficients in the chains and cochains of the cover.
\end{remark}

\subsection*{An $h$-principle for strata of differentials} We establish the above results through an $h$-principle for the homology of strata of differentials in  a range that increases as the number of simple zeros increases.  Specifically, we introduce the following topological moduli space that functions as an approximation to the stratum of abelian differentials.

\begin{defn}
	Let $\mu: [n] \to \bbN_{> 0}$ be an integer partition of $2g-2$.   We define $\cH_{g,\mu}^{\rm top}$ to be the topological stack parameterizing the data of:

\begin{enumerate} \item $C \in \cM_g$  a genus $g$ curve, \\
				\item  $D = \underset{i, \mu(i) > 1} \sum \mu(j) p_j$ a divisor in $\Hilb^{2g-2 - m_1(\mu) }  (C)$, \\
				
				\item  and $s \in \Gamma^{\rm top}(C, J^1 (\Omega^1_{C}(-D)))$ a nonvanishing \textbf{continuous} section of the first jet bundle of $\Omega^1_C(-D)$, such that $s_0(p) \neq 0$ for all $p$ in the support of $D$, where $s_0$ denotes the $0$th part of $s$.  
\end{enumerate}
  $\cH_{g,\mu}^{\rm top}$ is a fiber bundle over the stratum of the symmetric power of the universal genus $g$ curve that parameterizes choices of $C,D$ satisfying the above conditions.  (This stratum is $\cM_{g,n-m_1(\mu)}/\bS_\mu$, where $\bS_{\mu}$ is the Young subgroup associated to $\mu$).    
\end{defn}
 
 
 There is a canonical map $$\cH_{g,\mu} \to \cH_{g,\mu}^{\rm top}, \qquad (C, \omega) \mapsto (C, \div(\omega)^{> 1} , j^{1}(\omega)),$$ given by taking an abelian differential $\omega$   to the associated divisor $\div(\omega)^{>1}$ consisting of zeros of $\omega$ of order $> 1$  and the first jet of $\omega$ considered as a section of $\Omega^1(-D^{>1})$.    Our $h$-principle states:

\begin{thm}\label{comparison_theorem}
	 $H_i(\cH_{g,\mu}, \bbZ) \to H_i(\cH_{g,\mu}^{\rm top}, \bbZ)$ is an isomorphism for all $i \leq \mainbound.$
\end{thm}
With Theorem \ref{comparison_theorem} in hand, it becomes possible to analyze the homology of $\cH_{g,\mu}$ using the topological methods developed in the wake of the Madsen--Weiss theorem: in particular the results of Randall--Williams on homological stability for moduli spaces of surfaces with tangential structure \cite{RW}, and  of Galatius--Madsen--Tillmann--Weiss  on the homotopy type of the cobordism category \cite{GMTW}.

\subsection*{Stabilization maps}  So far we have discussed a ``stable range" for (co)homology, without mentioning any stabilization maps between homology groups.   Indeed, analogously to the situation for $\cM_g$, we will only construct natural stabilization maps after passing to curves with a framed marked point. Let $\cH_{g,\mu}^1 \to \cH_{g,\mu}$ be the moduli stack parameterizing the data of $(C,\omega) \in \cM_g$ together with an additional framed point $q \in C$ such that $\omega(q) \neq 0$.    

Given an integer partition $\mu$ of $2g-2$, let $\mu+1 +1$ denote the integer partition of $2g$ with two additional ones.  There is a homotopy class of stabilization maps $\cH_{g,\mu}^1 \to \cH_{g+1,\mu + 1 + 1}^1$, defined  as follows.  

Fix a flat torus $T$ with marked point $p$ and framed point $p'$.  Consider  $(C, \omega, q) \in \cH_{g,\mu}^1$  as a translation surface, and use the tangent vector at $q$ to cut parallel $\epsilon$-slits in $C, T$ starting at $p,q$ respectively. Then glue the opposite boundaries of each slit together to obtain a surface $T \cup C$.  The flat structure extends to $T \cup C$, with two additional singularities at the ends of the slit, corresponding to two additional simple zeros of a holomorphic form.   We let $(T \cup C, \omega, p') \in \cH_{g,\mu + 1 + 1}^1$ to be the image of $(C, \omega, q)$ under the stabilization map.  

\begin{remark}Formally, we take the domain of the stabilization map to be the subspace  $\cH_{g,\mu}^1 \times \bbR_{> 0}$  parameterizing $(C,\omega, q)$ and $\epsilon > 0$ which is less than the injectivity radius of $C$ over $2$---this space is canonically equivalent to $\cH_{g,\mu}^1$ by forgetting $\epsilon$.  
\end{remark}
 
 Theorem \ref{comparison_theorem} generalizes to a similar $h$-principle for pointed framed mapping spaces, and using this we show.
 
\begin{thm} \label{stabilizationtheorem}
	 $H_i(\cH_{g,\mu}^1) \to H_i(\cH_{g+1,\mu+1 + 1}^1)$ is an isomorphism for all $i \leq \mainbound -1.$
\end{thm}

\subsection*{Higher order differentials}
	Our $h$-principle also applies to strata of higher order differentials $(\Omega^1_C)^{\otimes k}$.  In fact, the case $k = 1$ is significantly more difficult than the other cases: because $(\Omega^1_C)^{\otimes k}$ is $(k-1)(2g  - 2)$-jet ample for all curves $C \in \cM_g$, it is possible to quickly obtain the case $k > 1$ using the method of Aumonier-Das \cite{AumonierDas}.   
	

\subsection{Proof method and structure of paper}

	Our argument for Theorem \ref{comparison_theorem} follows the typical structure of applications of Vassiliev's method to $h$-principles. We construct a finite dimensional approximation to the space of continuous sections and use Poincar\'e duality to reduce to proving that a Gysin map of compactly supported cohomology groups is an isomorphism in a range.  Then, using a simplicial resolution of a discriminant locus (in this case sections of $\Omega^1$ that have zeros to higher order than prescribed), we construct spectral sequences converging to the compactly supported cohomology of both sides, such that the natural map between these spectral sequences an isomorphism in a range of degrees.   The main difficulty we face is in constructing a partial resolution of the space of algebraic sections that both converges and has computable cohomology in a stable range.

	In \S \ref{sec:poset}, we consider the poset of divisors in the universal curve of multiplicity $\geq 2$ and associated Bar constructions and spectral sequences. 
	In  \S \ref{sec:dimbound}, we state the geometric bounds that give a spectral sequence converging to the compactly supported cohomology of $\cH_{g,\mu}$ in a range.  In  \S \ref{sec:hprinciple} we prove Theorem \ref{comparison_theorem}, and in \S \ref{sec:stablecomputation} we apply it to establish the other results of this paper.

\subsection{Relation to other work} \label{subsec:relationtootherwork}
 Theorem \ref{comparison_theorem} is in the spirit of recent results of Aumonier \cite{Aumonier}, Aumonier--Das \cite{AumonierDas}, and Das--Tosteson \cite{DasTost} on comparing spaces of holomorphic maps and holomorphic sections to spaces of continuous sections.  
 
A major obstacle to proving Theorem \ref{comparison_theorem} is that it cannot be reduced to a fiberwise statement (relative to the projection to $\cM_{g,n-m_1(\mu)}/\bS_\mu$).   Indeed, there are choices of $C, D'$ such that the space of nonsingular holomorphic sections of $\Gamma(C,\Omega^1(-D'))$ fails to approximate the homology of $\Gamma^{\rm top}(C, J^1(\Omega^1(-D')))$. Any proof must establish that  the locus of ``obstructed fibers" does not contribute to the stable cohomology.  In particular, the strategy of  Aumonier--Das  \cite{AumonierDas} does not apply.
 Instead, we use a variant of Das--Tosteson's framework, working globally over $\cM_{g,n-m_1(\mu)}/\bS_\mu$.  We formulate a criterion for convergence of the resolution in terms of certain dimension bounds on the locus of bad fibers.   To obtain the range in Theorem \ref{comparison_theorem}, we use geometric input:  Clifford's theorem, and the fact codimension of the locus of $d$-gonal curves grows in $g$ (see Proposition \ref{KeyGeometricBound}).
 

As mentioned in Remark \ref{CL}, simultaneous independent results of Chen--Larson \cite{ChenLarson} also give Theorem \ref{mainComputation}.  Chen--Larson's stable range is $g/6$ larger, and their method of proof is more direct.  On the other hand, our approach via the $h$-principle (Theorem \ref{comparison_theorem}) applies to integral cohomology and allows us to exhibit stabilization maps that induce isomorphisms on homology.  Substituting a dimension bound that Chen--Larson establish into our argument improves our stable range (see Remark \ref{improvedrange}).

\subsection{Further Questions}

\subsubsection*{Optimal ranges, and high codimension strata} Even for $H_0$ many strata lie outside of our stable range.   Indeed, any stratum in our stable range is connected, but Kontsevich and Zorich \cite{KontsevichZorich} have computed the connected components of all strata and shown that there can be up to three connected components.  Still, their answer displays  uniformity as the genus goes to $\infty$.   Is there a version of homological stability, or an $h$-principle that applies outside of the range considered in this paper?

\subsubsection*{Full computation of stable integral cohomology} As mentioned above, our methods reduce the computation of the stable integral homology to a homotopical problem.  More precisely, we show that the stable integral homology agrees with the homology of a space that fibers over a wreath product of cyclic groups, with fiber equivalent to the infinite loop space of a Thom spectrum.    Can one determine the stable homology explicitly?

\subsection{Acknowledgements} The author thanks Ronno Das for their joint collaborations using the Vassiliev method that have had a significant influence on this paper.  Thanks also to  Dawei Chen and Hannah Larson for sharing their related work and for helpful comments on a draft of this paper, and to Aaron Landesman for helpful conversations.

\subsection{Notation}

\begin{itemize}
	\item Unless otherwise indicated, we use $\dim X$ to denote $\dim_\bbC X$, the \textbf{complex dimension} of $X$.  
	\item For an integer partition $\mu$, we let $\ell(\mu) = \#\{i ~|~ \mu(i) > 0\}$, $m_j(\mu) = \#\{i ~|~ \mu(i) = j\}$, $\deg(\mu) = \sum_{i} \mu(i)$ and $\mu^{>1}$ be the integer partition of $\deg(\mu) - m_1(\mu)$ such that $\mu^{>1}(i) = \mu(i)$ if $\mu(i) > 1$ and $\mu^{>1}(i) = 0$ otherwise.  
	\item $\Gamma^{\top}(C,E)$ denotes the space of continuous sections of a map $p:E \to C$.  When $p, C,E$ are holomorphic, we let $\Gamma(C, E)$ be the space of holomorphic sections.   When $C,E$ are extended to families $\cC \to B, p:\cE \to \cC$ over a base space $B$,  we let $\uGamma^{\rm top}(C, E) \to B$ (resp. $\uGamma(C, E) \to B$) denote the space of relative continuous (resp.) holomorphic sections along $p$.  
	\item $\cC_g \to \cM_g$ denotes the universal curve and $\fH_2(\cC_g)$ denotes the Hilbert scheme of divisors of multiplicity $\geq 2$. 
	\item For an integer partition $\mu$, $W_\mu$ denotes the stratum of the Hilbert scheme of $\cC_g$ consisting of divisors of the form $\sum_{i} \mu(i) p_i$ for distinct $p_i \in C$.   
	\item $W_\mu/\fH_2(\cC_g)$ denotes the poset of pairs $(C, D) \in W_\mu$ and $(C,D') \in \fH_2(\cC_g)$ such that $D \subseteq D'$. And  $W_\mu/^\circ \fH_2(\cC_g)$ is the subposet where $D \subsetneq D'$.
\end{itemize}

\section{Posets of divisors of multiplicity $\geq 2$} \label{sec:poset}

There are several possible choices for which algebraic object to use to build the resolution of the discriminant locus (various topological posets, or topological categories).   We choose to adopt certain subposets of the Hilbert schemes of curves consisting of divisors whose multiplicity at every point is $\geq 2$.   In this section, we define the relevant posets, their combinatorial stratifications,  the associated spectral sequence for simplicial resolutions, and a convergence criterion for simplicial resolutions.   This material is closely analogous to \S3, \S4, \S5  of Das--Tosteson,  with slightly different (in fact simpler) combinatorics. We also strengthen the key convergence criterion, Theorem \ref{mainConvergence}, allowing for certain fibers to be obstructed.

\subsection{Definitions and combinatorics} \label{subsec:combinatorics}

For an integer partition $\mu$, the poset that we use to analyze $\cH_{g,\mu}$ will live over a stratum of the Hilbert scheme of the universal curve $\cC_g \to \cM_g$ defined as follows.

\begin{defn} 
	Let $W_\mu \subseteq \Hilb(\cC_g)$ be the locally closed substack parameterizing the data of 	\begin{itemize}
		\item a genus $g$ curve $C$ 
		\item a divisor $D \subseteq C$ of the form $D = \sum_{j} \mu(i) p_i$ for some disjoint points in $p_i \in C$.  
	\end{itemize}
\end{defn}

\begin{defn}
	For a curve $C$, we let $\fH_2(C) \subseteq \Hilb(C)$ be the closed subset of the Hilbert scheme parameterizing subschemes of length $\geq 2$ at every point of their support.  We let $\fH_2(\cC_g) \subseteq \Hilb(\cC_g)$ be corresponding closed substack of the Hilbert scheme of the universal curve $\cC_g \to \cM_g$.  
	We let $$W_\mu / \fH_2(\cC_g) := \{(C,D) \in W_\mu, (C, D') \in \fH_2(\cC_g) ~|~ D' \leq D \},$$  where we write $\leq$ for the containment relation on subschemes.   We let $W_\mu /^\circ \fH_2(\cC_g)$ denote  subposet where the inequality is strict.  For notational simplicity, we will sometimes write $\fH_2 = \fH_2(\cC_g)$ when $g$ is implicit.   
\end{defn}

\begin{remark}\label{notationWarning}
	In the notation of Das--Tosteson \cite{DasTost} $W_\mu / \fH_2(\cC_g)$ is denoted by $W_\mu  \leq \fH_2(\cC_g)$.  The subposet  $W_\mu /^\circ \fH_2(\cC_g)$  was denoted by $W_\mu  < \fH_2(\cC_g)$ in Das--Tosteson.
\end{remark}

The morphism of stacks $W_\mu / \fH_2(\cC_g) \to W_\mu$ is representable, and for any map from a scheme $Y \to W_\mu$ the pullback to $Y$ carries the structure of a poscheme over $Y$ (with relation induced by containment).  In our applications of this poset, we will pass to an arbitrary open subset  $U$ (in the classical topology) of a finite \'etale cover of $W_\mu$ by a scheme.  Then  $U/\fH_2(\cC_g) := U \times_{W_{\mu}} (W_\mu /\fH_2(\cC_g))$ is a topological space, naturally carrying the structure of a poset via the closed relation: 
\[\text{$(u \leq x)  \leq (u' \leq x')$ if and only if $u = u'$ and $x \leq x'$  }.\] ($U/\fH_2(\cC_g)$ is an open subset of a poscheme).    
A reader who wishes to avoid stacky issues may always work with such a $U$.  
For $u \in U$, we write $(C_u, D_u) \in W_\mu$ for the corresponding pair of a curve and  divisor. 

We  use the following combinatorial decomposition of the less than or equal to relation inside of $\Hilb(\cC_g) \times \Hilb(\cC_g)$.  

\begin{defn}
	Given a multiset of pairs of integers  $T = \{n_1 \leq m_1, \dots, n_k \leq m_k\}$, we let $\cN_T$ denote the substack of $\Hilb(\cC_g) \times_{\cM_g} \Hilb(\cC_g)$ consisting of a curve $C \in \cM_g$ and  pairs of divisors $(D,E) \in \Hilb(C) \times \Hilb(C)$ of the form $(\sum_{i} n_i p_i , \sum_{i} m_i p_i)$ for some choice of distinct $p_1, \dots, p_k$.  
\end{defn}

By definition, we have that $W_\mu / \fH_2 = \bigcup_{T} \cN_T$ where $T$ ranges over all multisets satisfying the following combinatorial conditions:
\begin{enumerate}
	\item if $j$ appears $r$ times in $\mu$, then there are exactly $r$ elements in $T$ of the form $j \leq k_i$ for some $k_1, \dots, k_r \in \bbN$,
	\item  every other element of $T$ is of the form $0 \leq m$ for some $m \geq 2$.
\end{enumerate}
Each $\cN_T$ is a locally closed substack of the Hilbert scheme:  $\cN_{\{n_i < m_i\}_{i = 1}^k}$ is in the closure of $\cN_{\{n_i' < m_i'\}_{i = 1}^j}$ if and only if there is a surjective function $f: [j] \to [k]$ such that $n_i= \sum_{ l \in f\inv(i)} n_{l}'$ and $m_i= \sum_{ l \in f\inv(i)} m_{l}'$ for all $i$.

\begin{defn}\label{partialOrder}
	Let $T = \{n_i < m_i\}_{i = 1}^j, S =  \{n_i' < m_i'\}_{i = 1}^{k}$ be multisets.  We write $T \leq S$  if there is an injection $f: [j] \to [k]$ such that $n_{f(i)} = n_{i}$ for all $i$ and $m_i \leq m_{f(i)}$ for all $i \in [j]$.  We write $T \leq_{+} S$ if there is a surjection $f: [k] \onto [j]$ such that $n_{i} = \sum_{i' \in f\inv(i)} n_{i'}$ and $m_{i} \leq \sum_{i' \in f\inv(i)} m_{i'}$.  
\end{defn}

\begin{defn}\label{def:essential}
 	Fix a partition $\mu$.  We say that a multiset of $T$ pairs of integers satisfying conditions (1) and (2) is \textbf{essential} if every element of type (1) is of the form $j \leq j + 1$ or $j \leq j$ and every element of the type (2) is of the form $0 \leq 2$.     If $T$ is essential, an element of $\cN_T$ will be called a \textbf{essential element} of $W_\mu/\fH_2$.
\end{defn} 
In terms of the poset, the strata $\cN_T$ corresponding to essential multi-set are important because they are minimal.

\begin{prop}\label{posethomologycomputation}
Let  $(C, D < D') \in \cN_T \subseteq W_\mu/^\circ \fH_2$.  This tuple satisfies $$ H^*(\rN(D, D'] ,\rN(D, D')) = \tilde H^*(\rN(D, D'))[-1] = \begin{cases} \bbZ[-r(T)+1]  & \text{ if $T$ is essential} \\ 0 & \text{ otherwise} \end{cases}$$ Here $\rN$ denotes the nerve of a poset,  $r(T)$ is the length of the longest chain in the interval $(D',D]$, and $\bbZ[-k]$ denotes the graded abelian group $\bbZ$ placed in cohomological degree $k$. \end{prop}
\begin{proof}
	The first equality is true for all intervals $D < D'$ , because $\rN(D, D']$ is contractible since $(D,D']$ has a top element, and so the cohomology of the pair is the reduced cohomology of the suspension of $\rN(D, D')$. 
	
	Suppose that $T = \{m_i \leq n_i\}$.  Then as a poset, the interval $[D,D']$ is isomorphic to $\prod_{i} ([m_i, n_i] \cap \{0, 2, 3 \dots\})$, which is a product of linear orders.  $T$ is essential if and only if each $([m_i, n_i] \cap \{0, 2, 3 \dots\})$ has $1$ or $2$ elements. Then the statement of the proposition follows from the computation of the poset-homology for intervals in a product of linear orders.  (See \cite[Example 1.1.1]{wachs2006poset} for the case where $T$ is essential, and the poset is boolean).
	  \end{proof}

\begin{remark} The reduced cohomology of  $\rN(D, D')$ is a categorification of the M\"obius number $\mu(D,D')$ that arises in spectral sequences that are built out of the poset, below.  (More precisely, this cohomology shifted by one is a categorification of the M\"obius number).
\end{remark}

\subsection{Spectral sequence associated to a continuous stratification}

We follow the conventions of  \S2 of Das--Tosteson \cite{DasTost} for topological posets, continuous stratifications, and Bar constructions.  (See Remark \ref{notationWarning} for one key notational difference).    In particular,  a \textbf{continuous stratification by the topological poset $U/\fH_2^{\leq N}(\cC_g)$} of a space $X$ equipped with a map $X \to U$  is a closed subset $Z \subseteq X \times_U (U/\fH_2^{\leq N}(\cC_g))$ satisfying the condition that if $p \leq q$ for $p,q \in U/\fH_2(\cC_g)$ then the fiber $Z_p$ contains the fiber $Z_q$.  The \textbf{Bar construction} associated to this stratification, denoted $\rB(U, \fH_2^{\leq N}(\cC_g), Z)$ is the semi-simplicial space with $[r]$ simplices the following space $$\rB(U, \fH_2^{\leq N}(\cC_g), Z)_r := \{u \in U, p_0 < \dots < p_r \in u/^{\circ} \fH_2^{\leq N}(\cC_g), z  \in Z_{p_r}\},$$ topologized as a subspace of a product of copies of $U/^{\circ}\fH_2^{\leq N}(\cC_g)$ and $Z$.

\begin{remark}
	For any topological poset $P$ over $U$, we may define a continuous stratification of $X$ by $P$ (with $P$ replacing $\fH_2^{\leq N}(\cC_g)$ above), and also extend the definition of Bar construction.  See \cite[\S2]{DasTost} for more details.  
\end{remark}

\begin{remark}\label{colimitInterpretation} Roughly, one can think of the Bar construction $\rB(U, \fH_2^{\leq N}(\cC_g), Z)$ as a model for a homotopy colimit over $U/^\circ \fH_2^{\leq N}(\cC_g)$ of $p \mapsto Z_p$.   The idea behind using this simplicial space as a resolution  is that this colimit approximates the union $\bigcup_{p \in U/^\circ \fH_2^{\leq N}(\cC_g)} Z_p$. 

Since we will be taking compactly supported cohomology, we do not actually use this perspective, though compactly supported cochains on (a truncation of) $\rB(U, \fH_2^{\leq N}(\cC_g), Z)$ can be interpreted as a certain homotopy limit of compactly supported cochains (or even as compactly supported cohomology a certain topological prestack)

\end{remark}

Following \cite[\S5.1]{DasTost} we construct a spectral sequence converging to compactly supported cohomology of the realization of the Bar construction.  

Let $N \in \bbN$. Let $X$ be a locally compact space with a map $X \to U$.  Let  $Z  \subseteq X \times_U U/\fH_2(\cC_g)$ be a continuous stratification of $X$ by $U/\fH_2(\cC_g).$  We stratify the realization Bar construction $|\rB(U,\fH_2(\cC_g)^{\leq N}, Z)|$ using the partial order $\leq_+$ on (see Definition \ref{partialOrder})

More precisely, given a multiset $T$ satisfying conditions (1) and (2) of \S \ref{subsec:combinatorics} and $|T| \leq N$, we define $(F_T)_r:= $ $$ \{u \in U, p_0 < \dots < p_r \in u/^{\circ} \fH_2^{\leq N}(\cC_g), z  \in Z_{p_r} \text{ such that } p_r \in \cN_S, S \leq_{+} T\}.$$  Each $(F_T)_\bdot$ is a closed semisimplicial subspace of $\rB(U,\fH_2(\cC_g)^{\leq N}, Z)$ because $\cup_{S \leq_{+} T} \cN_T$ is closed and downwards closed.  Moreover, if $T_1 \leq_{+} T_2$, then $F_{T_1} \subseteq F_{T_2}$.   Passing to realizations, we obtain a stratification of $|\rB(U,\fH_2(\cC_g)^{\leq N}, Z)|$ by the poset of multisets.  Associated to this stratification, there is a compactly supported cohomology spectral sequence with total $E_1$ page $$\bigoplus_{T} H^*_c\left(|F_T| - \bigcup_{S <_{+} T} |F_{S}|\right),$$  converging to the compactly supported cohomology of the Bar construction, where the sum ranges over all $T$ satisfying the above conditions. (See \cite[\S 2.1]{DasTost} for a discussion using our conventions on compactly supported cohomology spectral sequence). 

The assignment  $(u < p_0 < \dots < p_r, z,  x \in \overset{\circ~}{\Delta^r}) \mapsto (u < p_r, z)$, defines a continuous map $\pi:(|F_T| - \cup_{S <_{+} T} |F_{S}|) \to Z|_{ U \times_{W_\mu} \cN_T}$, exhibiting $(|F_T| - \cup_{S <_{+} T} |F_{S}|)$ as a fiber bundle,  pulled back from $ U \times_{W_\mu} \cN_T$.  Given  $(u < p, z) \in  Z|_{U \times_{W_\mu} \cN_T}$ the fiber of $\pi$  only depends on the pair of divisors $D_u < D'$  corresponding to $u< p$: it is given by $\rN(D_u, D'] - \rN(D_u, D')$.    Thus the compactly supported cohomology of the fiber is the cohomology of the pair $(\rN(D_u, D'] ,\rN(D_u, D'))$, computed in Proposition \ref{posethomologycomputation}.  Therefore $H^*_c((|F_T| - \cup_{S <_{+} T} |F_{S}|)) \iso H_c^*(Z|_{\cN_T \times_{W_\mu}  U/\fH_2(\cC_g)})[-r(T)+1]$ if $T$ is essential, and zero otherwise. So the compactly supported cohomology spectral sequence takes the form given in the following theorem.

\begin{thm}
\label{spectralSequence}
	Let $N,U,X,Z$ be as above.  There is a spectral sequence converging to the compactly supported cohomology of the realization $|\rB(U, \fH_2(\cC_g)^{\leq N}, Z)|$ with total $E_1$ page given by $$\bigoplus_{p,q} E_1^{p,q} = \bigoplus_{T \text{ essential } \cN_T \subseteq  W_\mu/^\circ(\fH_2^{\leq N}(\cC_g))} H_c^*(Z|_{\cN_T \times_{W_\mu}  U/\fH_2(\cC_g)})[-r(T)+1],$$  where $r(T)$ is defined as in Proposition \ref{posethomologycomputation}.	\end{thm}

Because the spectral sequence is constructed from the compactly-supported cohomology spectral sequence by a filtration, it is functorial for Gysin maps.  See the discussion in \cite[\S 2.1]{DasTost}.

\subsection{Convergence criterion}
	
	 Let $U$ be a semi-algebraic open subset of a finite \'etale cover of $W_\mu$.  Let $X \to U$ be a semi-algebraic map. Let $Z \subseteq X \times_{U} U/\fH_2(\cC_g)$ be a semi-algebraic closed subset, defining a continuous stratification of $X$ by $\fH_2(\cC_g)$ satisfying the following additional property: \begin{equation}\label{intersectionAssumption} Z_{C_u,D_u \leq D_1'} \cap Z_{C_u,D_u \leq D_2'} =  Z_{C_u,D_u \leq \max(D_1', D_2')},\end{equation} for all $u \in U$ and $D_1', D_2' \geq D_u$.   Associated to such a stratification, we let ${\rm Disc}_Z$ be the image in $X$ of the projection of $Z|_{U/^{\circ}\fH_2(\cC_g)}$.    Because this projection factors through a proper map, ${\rm Disc}_Z$ is closed.
	 
	  We now give a criterion that shows that under certain dimension bounds, the compactly supported cohomology of ${\rm Disc}_Z$ can be described in terms of the bar construction of a truncation of $U/\fH_2(\cC_g)$ in a range.  To state it, we introduce some additional notation. 
	
	\begin{defn} Given a multiset $T$ consisting of pairs $n_t \leq m_t$ for $t \in T$, we let $$\alpha(T) := \sum_{t \in T} (n_t - m_t) - \#\{ t ~|~ n_t = 0\}.$$ We let $r(T)$ be the length of the longest chain in the interval $(D_u , D]$ for any $D_u \leq D \in N_T$.   Recall also that $\deg(T) := \sum_{t} m_t$.    
	\end{defn}

    \begin{defn} Fix a dimension $e \in \bbN$.	Let $f: A \to B$ be a proper map between locally compact semi-algebraic spaces.  We say that the induced map on compactly supported cohomology is \textbf{connected in codimension $\geq I$ relative to $e$}  if we have that $H^{e - i}_c(A) \leftarrow H^{e - i}_c(B)  $ is an isomorphism for all $i < I$ and a surjection for $i = I$. \end{defn}  
	
	\begin{thm}\label{mainConvergence} Let $N \geq \deg(\mu) + 2$.  
		Suppose that for every stratum $\cN_T$ of $U/^\circ \fH_2^{\leq N}(\cC_g)$ associated to an essential multiset $T$ we have that $\codim_{\bbC} Z|_{\cN_T}:= \dim_{\bbC} X - \dim_{\bbC} Z|_{\cN_T}$ is greater than or equal to $ \alpha(T)$.  Then the map on compactly supported cohomology induced by $$\pi: \rB(U,  \fH_2^{\leq N}(\cC_g),  Z) \to {\rm Disc}_Z$$ is connected in codimension $\geq \lfloor \frac{N -\deg (\mu)}{2} \rfloor + 1$ relative to $\dim_{\bbR} X$.  
		(Here $\deg (\mu) = \sum_{i} \mu(i)$ is the degree of a divisor in $W_\mu$).  
	\end{thm}
	\begin{proof}
	
	We define the following collection of multisets $\fM_N$. We put $T \in \fM_N$ if  $\cN_T \in W_\mu/\fH_2^{\leq N}$ and $T$ is maximal in the following sense: there does not exist  an essential multiset $T'$ with $\cN_{T'} \in W_\mu/\fH_2^{\leq N}$ and $T' > T$. 

	The rough idea of the proof is to decompose $\Disc_Z$ into pieces such that the fibers of $\pi$ are either (1) contractible or (2) have dimension bounded in terms of  $Z|_{\cN_T}$ for $T \in \fM_N$.  Accordingly we use the stratification of $\Disc_Z$ provided by the following lemma.
	\begin{lem}\label{BreakupLemma}
		There is a finite semi-algebraic triangulation of $\Disc_Z$ such that every locally closed triangle $\sigma \subseteq \Disc_Z$ satisfies one of the following.
		\begin{enumerate}[(I)]
			\item For every $(C_u,D_u, s) \in \sigma$, the poset $$Q_s :=\{(C_u, D_u < D' ) \in U/\fH_2^{\leq N} \text{ essential } ~|~ s \in Z_{C_u, D_u, D'},\}$$ has a maximum element.
			\item  For every $(C_u,D_u, s) \in \sigma$ and any $(C_u, D_u < D' ) \in Q_s$  there exists $(C_u, D_u < D'' ) \in Q_s$ such that $D'' \geq D' > D_u$, and $D_u < D'' \in \cN_T$ for $T \in \fM_N$
	\end{enumerate}
	\end{lem}
	
	From the compactly supported cohomology spectral sequence associated to the triangulation of Lemma \ref{BreakupLemma}, it suffices to show that for every locally closed triangle $\sigma$ that either (1) $H^i_c(\sigma) \to H^i_c(\pi\inv(\sigma)$ is an isomorphism for all $i$ or (2) $H^i_c(\sigma) = H^i_c(\pi\inv(\sigma)) = 0$ for all $i \geq \dim_{\bbR} X - \lfloor\frac{N - \deg (\mu))}{2} \rfloor - 1$. ((1) and (2) correspond to cases (I) and (II) respectively).   
	
	In case (I), the fiber of every element $(C_u, D_u, s) \in \sigma$ is identified with bar construction of the poset $P_s$ consisting of $D' \geq D_u$ such that $s \in  Z_{C_u, D_u, s}$ and $\deg(D') \leq N$.  We claim that this poset is finite and its bar construction (or nerve) retracts to $$Q_s=\{(C_u, D_u < D' ) \in U/\fH_2^{\leq N} \text{ essential } ~|~ s \in Z_{C_u, D_u, D'}\}.$$ This  suffices to handle case (I), since the map $\pi\inv(\sigma) \to \sigma$ has contractible fibers (because $Q_s$ has a maximum) and so induces an isomorphism on compactly supported cohomology by proper base change.   To show the claim, we write $\eta(D')$ for the largest element such that  $D \leq \eta(D') \leq D'$ and $D \leq \eta(D')$ is essential.  Then $\eta|_{{U/\fH_2}^{\leq N}}$ has finite fibers, and because $Q_s$ is finite (consisting of the finitely many elements $\leq$ to the maximum) so is $\eta\inv(Q_s)$.  Finally, by definition, $\eta$ is right adjoint to the inclusion of $Q_s$ into $P_s$, hence induces a deformation retraction on nerves.  	
	
	In case (II), we bound the compactly supported cohomological dimension of $\pi\inv(\sigma)$ and $\sigma$.  Restricting the continuous stratification of $X$ to $\sigma$, and applying Theorem \ref{spectralSequence} we obtain a spectral sequence converging to the compactly supported cohomology of $\sigma$ which is a sum of terms of the form $H^*_c((Z \cap \sigma)|_{\cN_S}) [-(r(S)-1)]$ where $S$ ranges over essential multisets of degree $\leq N$.  Therefore, it suffices to bound $\dim_{\bbR} (Z \cap \sigma)|_{\cN_S} + r(S) - 1$ and $\dim(Z \cap \sigma)$.
	
	We first bound the dimension of $(Z \cap \sigma)|_{\cN_S}$.  We have that $$\dim_{\bbR} (Z \cap \sigma)|_{\cN_S} \leq \max_{T \in \fM_N } \dim_{\bbR} Z|_{\cN_T}.$$
	Indeed, by definition,  in case (II) there is a surjection onto $(Z \cap \sigma)|_{\cN_S}$  from the semi-algebraic space 
	 $$ \bigsqcup_{T \in \fM_N} \{  C_u, D_u < D' \leq D'', s  ~|~ 
	 (C_u, D_u \leq D'', s) \in (Z \cap \sigma)|_{\cN_T}, ( C_u, D_u \leq D', s )\in (Z \cap \sigma)|_{\cN_S} \},$$
	   by forgetting $D''$. And there is finite to one map from the same space to $\underset{T \in \fM_N}\bigsqcup Z|_{\cN_T},$ given by forgetting $D'$. 
	  
	  It follows that $\dim_\bbR \sigma$ is also less than or equal to $\underset{T \in \fM_N } \max \dim_{\bbR} Z|_{\cN_T}$, because there is a surjection $\bigsqcup_{S \text{ essential }, \deg S \leq N } Z \cap \sigma|_{\cN_S} \twoheadrightarrow \sigma$.  (Indeed $\sigma \subseteq \Disc_Z$ so for every $(C_u, D_u, s ) \in \sigma$ there exists such $D' > D$ of degree $\leq N$ such that $s \in Z_{(C_u, D_u < D')}$ and we can always take $D = \eta(D')$ to be essential).
	  
	  By our hypotheses on $\codim_{\bbR} Z|_{\cN_T}$ we have $$\max_{T \in \fM_N}(\dim_{\bbR}( Z|_{\cN_T}) + r(T) - 1 )\leq \dim_{\bbR} X - 1 -  \min_{T \in \fM_N } (2\alpha(T) - r(T))).$$ To analyze the quantity $\underset{T \in \fM_N} \min (2 \alpha(T) - r(T))$, suppose that $T$ has $a$ elements of the form $\mu(i) \leq \mu(i)$, $b$ elements of the form $\mu(i) \leq \mu(i)+1$ and $c$ elements of the form $0 \leq 2$.  Then $\alpha(T) =  b + c$, $r(T) = b + c$, and $\deg(T) = \deg(\mu) + b + 2c$.  If $T \in \fM_N$,  then $c = \lfloor \frac{N - \deg(\mu) - b }{2} \rfloor$ since otherwise we can find a  $T' \geq T$ by increasing $c$ by one. 
	  Since $b,c \geq 0$, in the worst case we have $b = 0$, and $c = \lfloor\frac{N- \deg(\mu)}{2}\rfloor$.  So it follows that $2\alpha(T) - r(T) = b + c \geq \lfloor\frac{N - \deg(\mu)}{2} \rfloor.$   Therefore the compactly supported cohomological dimension of both $\sigma$ and $\pi\inv(\sigma)$ is less than or equal to $\dim_{\bbR} X - \lfloor \frac{N - \deg(\mu)}{2}\rfloor - 1$, completing the second case.
	 \end{proof}
	
	\begin{proof}[Proof of Lemma \ref{BreakupLemma}]

		For $T$ essential, we let $J_T = \pi(Z|_{\cN_T})$.  Since $\pi, Z, \cN_T$ are semi-algebraic, we have that $J_T$ is also semi-algebraic.  Let $J_T^{\circ} := J_T - \bigcup_{ S> T, ~ {\rm essential} } \pi(Z|_{\cN_{S}})$.  
		
		We first show that for essential $T$ with $\deg(T) \leq N$, every $(C_u, D_u, s) \in J_T^\circ$ satisfies the condition of case (I).  Indeed, suppose that $(C_u, D_u, s) = \pi(C_u, D_u < D', s)$ for $(C_u, D_u < D') \in \cN_T$, then we claim that $(C_u, D_u < D')$ is a maximum of $Q_s$.  Let $(C_u, D_u < E)$ be any other element of $Q_s$.  Then by  \eqref{intersectionAssumption} we have that $s \in Z_{C_u, D_u < \max(E, D')}$.  If $\max(E, D') \neq D'$, this contradicts the assumption that $(C_u, D_u < D', s) \in J_T^\circ$, so it follows that  $E \leq D'$.  
	
		Next, let $\Theta := \Disc_Z - \bigcup_{T ~{\rm essential},~ \deg(T) \leq N} J_T^\circ$.  We show that every $(C_u , D_u , s) \in \Theta$ satisfies the condition of case (II).  Suppose $(C_u, D_u < E_1) \in Q_s$.  If $\type( D_u < E_1) \in \fM_N$ then we are done.  Otherwise, since $(C_u , D_u , s) \not \in J_{\type( D_u < E_1)}^\circ$ we can find an $E_2 \succ E_1$ with $(C_u, D_u < E_2) \in Q_s$:  there is a $T' > \type( D_u < E_1)$ and an $E_1'$ with 	$\type(C_u, D_u < E_1')  = T'$ and $s \in Z_{C_u, D_u < E_1'}$ and we take $E_2$ to be a minimal element satisfying $E_1 < E_2 \leq \max(E_1', E_1)$.  Continuing like this we obtain a chain of elements $E_1 \prec E_2  \prec  E_3 \dots $ and eventually $\type(E_r) \in \fM_N$.  Then $D'' = E_r$ works to establish the condition of case (II).
		
		Thus we have a partition of $\Disc_Z$ into two semi-algebraic subsets  $\Disc_Z - \Theta$  and $\Theta$ such that any element of the first subset satisfies condition of case (I) and any element of the second satisfies case (II).  By standard results on semi-algebraic sets (see \cite[Theorem 5.7]{DasTost}), we may choose a semi-algebraic triangulation of $\Disc_Z$  that is compatible with $\Theta$ in the sense that every open simplex is either contained in $\Theta$ or $\Disc_Z - \Theta$, finishing the proof.
	\end{proof}

\section{Dimension bounds} \label{sec:dimbound}
 
  Let $\uGamma(C, \Omega_C^1(-D))$ be the stack over $\Hilb(\cC_g)$ parameterizing $\{C \in \cM_g, D \in \Hilb^d(C), s \in \Gamma(C, \Omega_C^1(-D))\}$.  In this section, we bound the dimension of the restriction of  $\uGamma(C, \Omega_C^1(-D))$ to the following locus.   
 \begin{defn}
    We define the \textbf{obstructed locus} ${\rm Obs}_{g,d} \subseteq \Hilb^d(\cC_g)$ to be the closed substack of pairs where $\dim H^0(\Omega^1_C(-D)) > g-d$. 
 \end{defn} 
 
   The next proposition is the main algebro-geometric input into  Theorem \ref{comparison_theorem}.

 \begin{prop}\label{KeyGeometricBound}
 	     For $d \leq g$, the codimension of $\uGamma(C, \Omega_C^1(-D))|_{\Obs_{g,d}}$ in $\uGamma(C, \Omega_C^1(-D))$  is  $\geq \geometricBound .$  
\end{prop}
 \begin{proof}
  	We first observe that  $$\dim_\bbC {\rm Obs}_{g,d} = \dim_\bbC  ~ \{C \in \cM_g, D \in \Hilb_d(C) ~|~ h^0(\cO_C(D)) \geq 2 \} \leq 2d +2g - 4.$$  Indeed, since every such $C,D$ admits a subdivisor $E \subseteq D$ that arises from a map to $\bbP^1$, the dimension of this locus is bounded above by the following quantity, expressed in terms of the dimension of Hurwitz schemes of degree $\leq d$:   $$\max_{e \leq d}( \dim(\{C \in \cM_g, f: C \to \bbP^1,  \deg f = e\} - 2  + d -e  ).$$  (Here the term $d-e$ comes from the dimension of the space of divisors on $C$ containing $E$, and the $-2$ comes from reparameterizations of $\bbP^1$ fixing $0$).    Since the degree $d$ Hurwitz scheme has dimension $2d + 2g - 2$, the bound on $\dim_\bbC \Obs_{g,d}$ follows.
	
	Next, Clifford's theorem on special divisors implies that $h^0(\Omega_C^1(-D)) \leq g-d/2$.  So combining the bound on the dimension of the base and the fibers, we have that $$\dim_\bbC( \uGamma(C, \Omega_C^1(-D))|_{\Obs_{g,d}}) \leq 2d +2g - 4 + g-d/2.$$ On the other hand, the restriction of $\uGamma(C, \Omega_C^1(-D))$ to the complement of $\Obs_{g,d}$ has dimension $3g-3 + d + g-d$.   Subtracting these, we obtain a lower bound on the codimension of $\geometricBound$.  
 \end{proof}

 \subsection{Brill Noether bounds}
The bound in Proposition \ref{KeyGeometricBound} is not optimal.  Here we observe that universal bounds on the dimension of the space of $g_r^d$'s, can be used to improve   Proposition \ref{KeyGeometricBound}.  (This subsection is not used in the rest of the paper).
 
 \begin{defn} Let $d, g, r \in \bbN$.  We define $B(g, r ,d)$  to be $$ \dim_{\bbC} \{C \in \cM_{g},  D \in \Hilb^d(C),  h^0(C, \cO(D)) \geq r + 1  \}. $$   An \textbf{universal BN bound} is an upper bound on $B(g,r,d)$.
\end{defn}

\begin{prop}
	The codimension of $\uGamma(C, \Omega_C^1(-D))|_{\Obs_{g,d}}$ in $\uGamma(C, \Omega_C^1(-D))$  is  $ \geq  3g-3 + d   -  \underset{r = 1, \dots, d} \max(r + B(g,r,d)).$  
\end{prop}
\begin{proof}
It follows from the Riemann-Roch formula: $h^0(\Omega^1_C(-D)) = h^0(\cO(D)) - 1 + g -d $, that the codimension is bounded below by  $$\geq 3g - 3 + d + g-d -  \max_{r = 1, \dots, d}(r + g-d + B(g,r,d)).$$ Simplifying yields the proposition.
\end{proof}

\begin{remark}\label{improvedrange}
	Chen--Larson \cite{ChenLarson} give an argument that $B(g,r,d) \leq 2g + 2d - r - 4$ which implies the codimension of $\uGamma(C, \Omega_C^1(-D))|_{\Obs_{g,d}}$  is $\geq 3g-3 + d -(2g + 2d - 4) = g-d + 1$.  Substituting this improved bound into our argument in  \S \ref{sec:hprinciple} yields an improved range of $i \leq g/6 + \mainbound$. \end{remark}

\section{Proof of the $h$-principle} \label{sec:hprinciple}
	In this section, we prove Theorem \ref{comparison_theorem}.  First, using a simplicial resolution, we construct a spectral sequence converging to the compactly supported cohomology of $\cH_{g,\mu}$ in a range of degrees that grows as the number of simple zeros of $\mu$ grows.   Then we construct a similar resolution of a finite dimensional approximation of $\cH_{g,\mu}^{\top}$.  We show that the induced map between the spectral sequences associated to these two resolutions is an isomorphism in a range of degrees, from which we deduce Theorem \ref{comparison_theorem}.

\subsection{Resolving algebraic sections}  \label{sec:algConvergence}

Over $\fH_2(\cC_g)$ we have the universal stack $$\uGamma(C,\Omega^1_C(-D)) := \ker(\uGamma(C,\Omega^1_C) \to \uGamma(D, \Omega^1_C)),$$ whose fiber over $(C,D)$ is the space of sections of $\Omega^1_C$ vanishing at $D$.   This defines a continuous stratification of $\uGamma(C,\Omega^1_C)$ by the topological poset $\fH_2(\cC_g)$.    

To study the cohomology of the stratum $\cH_{g,\mu}$, we use the postack $W_{\mu^{>1}}/\fH_2(\cC_g)$. Note that there is a natural map $ \cH_{g,\mu} \to W_{\mu^{>1}}$ given by taking the divisor of a differential, and forgetting points of multiplicity one.     Given $U$ a (classical) open subset of a finite \'etale cover of $W_{\mu^{>1}}$ we will more generally consider the cohomology of the pullback $\cH_{g,\mu}|_U$ via the topological poset $U/\fH_2(\cC_g)$.  (In the proof of the main Theorem \ref{comparison_theorem} we only consider the case where $U$ is an open subset of a scheme).    Accordingly we define $Z|_U$ to be the pullback  of $\uGamma(C,\Omega^1_C(-D))$ along the map $$U/\fH_2(\cC_g) \to \fH_2(\cC_g) \qquad  (u, C_u, D_u\leq D') \mapsto (C_u, D').$$   So the fiber of $Z|_U$ over $(C_u, D_u \leq D') \in U/\fH_2(\cC_g)$ is the space of sections of $\Omega^1_C$ that vanish at $D'$.  

Write $E|_U$ for the pullback of $\uGamma(C,\Omega^1_C(-D))$ along the canonical map $U \to W_{\mu^{>1}} \subseteq \fH_2(\cC_g)$.   Thus $E|_U$ parameterizes tuples $(u \in U, s \in \Gamma(C_u, \Omega^1(-D_u))$, and $Z|_U \subseteq E|_U$ is a closed subset, defining a continuous stratification of $E|_U$ by $Z|_U$.   

We have that $\cH_{g,\mu}|_U \subseteq E|_U$ is the open subset obtained by removing the discriminant locus $\Disc_{Z|_U}= \bigcup_{(C_u,D_u < D') \in U/^\circ\fH_2(\cC_g) } Z_{C_u, D_u < D'}.$

Let $N = \lfloor 2 g/3 \rfloor$.  
It follows from Proposition \ref{KeyGeometricBound} that for every multiset $T$ such that $\cN_T \in U/\fH_2^{\leq N}$, we have $$\dim_{\bbC} E|_U - \dim_{\bbC} (Z|_U)|_{\cN_T} \geq \alpha(T).$$  Indeed, Proposition \ref{KeyGeometricBound} implies the codimension of the obstructed locus is positive, and the unobstructed locus has dimension $3g-3 + |T| + (g-1) - \deg(T)$, compared to $\dim_{\bbC} E|_U = 3g-3 + \#\{i ~|~ \mu(i) > 0\} + (g-1) - \deg(\mu^{> 1})$, where $r$ is the number of parts of $\mu^{>1}$.  Subtracting these gives that the codimension of the unobstructed locus is $\alpha(T)$.  

Therefore Theorem \ref{mainConvergence} implies that the Bar construction $\rB( U, \fH_2^{\leq N}, Z)$ approximates the compactly supported cohomology of the complement of $\cH_{g,\mu}$, as follows.

\begin{thm}\label{algebraicResolution}
 	Let $g \in \bbN$ and $\mu$ be an integer partition of $2g-2$. Let $e :=  2(3g-3 + \#\{i ~|~ \mu(i) > 0\} + (g-1) - \deg(\mu^{>1}))$.  Assume that $\lfloor 2 g/3 \rfloor \geq \deg(\mu^{>1}) + 2$.  Then $$H^i_c(\rB( U, \fH_2^{\leq \lfloor 2g/3 \rfloor}, Z)) \leftarrow H^i_c(\Disc_Z) $$ is connected in codimension $\lfloor \frac{\lfloor2g/3 \rfloor - \deg(\mu^{>1})}{2}\rfloor + 1$ relative to $e$.  
\end{thm}

By Theorem \ref{spectralSequence} we obtain a spectral sequence converging to the homology of spaces of algebraic sections in a range. 

\begin{cor}\label{algebraicSectionSpectralSequence}
	 Let $\kappa = \lfloor \frac{\lfloor2g/3 \rfloor - \deg(\mu^{>1})}{2}\rfloor$. Let $e :=  2(3g-3 + \#\{i ~|~ \mu(i) > 0\} + (g-1) - \deg(\mu^{>1}))$. There is a spectral sequence that converges to the compactly supported cohomology of the complement of $\cH_{g,\mu}|_U$ in $\uGamma(C, \Omega^1_C(-D))|_U$ in degrees $\geq e - \kappa$, with total $\bE_1$ page taking the following  form in degrees $\geq e - \kappa$: $$\bigoplus_{p + q \geq e - \kappa } \bE_1^{p,q} = \bigoplus_{T \text{ essential, }\cN_{T} \subseteq (W_{\mu^{>1}}/^\circ \fH_2^{\leq N})} ~ \bigoplus_{i \geq e- \kappa } H^{i + e - 2 \deg(T)}_c(\cN_T|_U).$$ 
\end{cor} 
Recall that each $\cN_T$ is isomorphic to a colored configuration space of $\cC_g$, so that each term $H^*_c(\cN_T)$ may be computed in terms of the cohomology of $\cM_{g,n}$ with its symmetric group action.   We will not attempt to compute the differentials of this spectral sequence directly; instead we will compare it with a spectral sequence related to the homology of $\cH_{g,\mu}^{\rm top}$, constructed below.  

\begin{remark}
	Using an augmented version of the Bar construction, we may obtain a spectral sequence converging directly to the cohomology of $\cH_{g,\mu}|_U$.  Since the cohomology of this subset and its closed complement differ by a long exact sequence, Corollary \ref{algebraicSectionSpectralSequence} suffices for our purposes.  
\end{remark}

\subsection{Finite dimensional approximations} The space $\cH_{g,\mu}^{\top}$ is an infinite dimensional fiber bundle over $W_{\mu^{>1}}$.  Since the results of Section \ref{sec:poset} only apply to compactly supported cohomology, we will use finite-dimensional approximations to relate the homology groups of $\cH_{g,\mu}^{\top}$ to compactly supported cohomology groups.   We follow the treatment in \cite{DasTost}[\S6.3] and \cite{Aumonier}[\S 7], via ideas that can be traced to \cite{Mostovoy}.  

Let $B$ be a smooth algebraic variety.  Let $p:\cC \to B$ be a smooth projective family of algebraic curves, and let $\cE \to \cC$ be an algebraic vector bundle.  We let $\uGamma^\top(\cC, \cE) \to B$ be the space of  relative continuous sections:  a point of  $\uGamma^\top(\cC, \cE)$ is determined by the data of $b \in B$ and a continuous section $s \in \Gamma^\top(p\inv(b),\cE|_{p\inv(b)})$, and $\uGamma^\top(\cC, \cE)$ is topologized as a bundle of Banach spaces.  (For any contractible open $V \subseteq B$, we may trivialize $\cE |_V \to \cC|_V$ to trivialize $\uGamma^\top(\cC, \cE)|_V$).  

\begin{prop}\label{fdApprox}
	There exists a sequence of finite dimensional semi-algebraic vector bundles $
	\cF^{(k)} \subseteq \uGamma^\top(\cC, \cE)$  over $B$ such that 
	\begin{enumerate}
		\item $\cF^{(k)} \subseteq \cF^{(k+1)}$ and $\cF^{(k)} \supseteq \uGamma^{\rm alg}(\cC, \cE)$ for all $k$,
		\item For any open subset $G \subseteq \uGamma^\top(\cC, \cE)$ we have that $\colim_{k} H_i(\cF^{(k)} \cap G) \iso H_i(G)$ for all $i \in \bbN$.  
	\end{enumerate}
\end{prop}
\begin{proof}
	Embed $\cC$ into $\bbP^r_B$ and let $M$ be the restriction of $\cO(1)$.  We take $\cF^{(k)}$ to be the image of the following composite \[\phi_k: \uGamma^\hol(\cE \otimes M^{\otimes k}) \otimes \uGamma^{\antihol}(\overline{M}^{\otimes k})   \to  \uGamma^\top(\cE \otimes  M^{\otimes k} \otimes \overline M^{\otimes k}) \to^{(\eta\inv)^{\otimes k}} \uGamma^\top(\cE)\]
Here \(\Gamma^{\hol},\Gamma^{\antihol}\)  respectively denote holomorphic and anti-holomorphic sections, and $\overline M$ is the anti-holomorphic line bundle obtained by complex conjugating the transition maps of $M$.   The first map includes (anti-)holomorphic sections into continuous sections and multiplies them pointwise.   The second map uses the inverse to the semi-algebraic trivialization $\eta: \bbC \times \cC \to M \otimes \overline M$ defined by $\eta(1) = \sum_{i = 0}^r z_i \otimes \overline z_i $, where $z_i \in \Gamma(\cC, M)$ are restrictions of the homogeneous coordinates of $\bbP^r$. 

By \cite[Lemma~6.17]{DasTost}, the map $\phi_k$ is injective.   Because $M$ is ample, $\uGamma^\hol(\cE \otimes M^{\otimes k})$ is a vector bundle for $k$ sufficiently large.  Therefore, reindexing if necessary, we may guarantee that every $\cF^{(k)}$ is a subbundle.   As in the discussion immediately after \cite[Lemma~6.17]{DasTost}, we have that $\cF^{(k)} \subseteq \cF^{(k+1)}$, $\cF^{(k)}$ contains all of the holomorphic (hence algebraic) relative sections of $\cE$,  and $\cF^{(k)}$ is semialgebraic, establishing (1).  

It follows from  the Stone--Weierstrass theorem, as applied in the proof of \cite[Lemma~7.2]{Aumonier}, that for every $b \in B$, the union $\bigcup_{k} (\cF^{(k)})|_b$ is dense in $\Gamma^{\top}(\cC_b, \cE_b)$.   For any open subset $V \subseteq B$ on which $\cC|_U, \cE|_V$ may be continuously trivialized, we have a trivialization of $\uGamma(\cC,\cE)|_V$.   Therefore by \cite[Proposition 6.16]{DasTost} ${\rm hocolim}_{k} (\cF^{(k)} \cap G|_V) \simeq G|_V$. (Note that in \cite[Proposition 6.16]{DasTost}  $B$ is used to denote a Banach space fiber, not the base).   Using a hypercovering of $B$ by contractible $V$, it follows that ${\rm hocolim}_{k} (\cF^{(k)} \cap G)$ is weakly equivalent to $G$ and passing to homology, we obtain (2).  (For the statement on homology only, it suffices to use a Mayer-Vietoris spectral sequence instead).  
\end{proof}

\subsection{Resolving continuous sections} \label{sec:topRes} For spaces of continuous sections, we build a similar resolution to the case of algebraic sections.  


Over $W_{\mu^{> 1}} \times_{\cM_g} \cC_g$, we have a family of vector bundles whose fiber above $(C,D)$ is the bundle $J^1(\Omega^1(-D))$ on $C$.  Let $E^{\rm top} := \uGamma^{\rm top}(C, J^1(\Omega^1(-D))) \to W_\mu$ be the space of relative continuous sections.   

 \begin{defn} 
  Given $(C, D \leq D')$, we let $Z^{\top}_{C, D \leq D'} \subseteq E^\top_{C,D}$ be the subspace of sections $(s_0,s_1) \in \uGamma^{\rm top}(C, J^1(\Omega^1(-D)))$ such that: 
 \begin{enumerate}
	\item if $D'-D$ has multiplicity $1$ at $q \in C$, then $s_0(q) = 0$, and
 	\item if $D' - D$ has multiplicity $> 1$ at $q \in C$, then $s_0(q) = s_1(q) = 0$.
 \end{enumerate}
 \end{defn}
 
 \begin{prop}
 	$Z^{\top} \subseteq E^{\top}$ is closed, hence defines a continuous stratification of $E^{\top}$ by $W_{\mu^{> 1}}/\fH_2(\cC_g)$.
 \end{prop}
 \begin{proof}
 	First we observe that the assignment $D \leq D' \mapsto ((D'-D)^{= 1} , (D'-D)^{>1})$, extracting the multiplicity $1$ and multiplicity $>1$ part of $D'-D$  defines a continuous map $$W_{\mu^{> 1}}/\fH_2(\cC_g) \to \left(\bigsqcup_{r \leq \ell(\mu^{>1})}   {\rm UConf}_r(\cC_g) \right)\times_{\cM_g} \Hilb(\cC_g).$$ Indeed if $D'-D$ has multiplicity $1$ at a point $p$ of $C$, then because $D' \in \fH_2(C)$ has multiplicity $\geq 2$ at every point, we have that $\mult_p(D) \leq \mult_p(D')$ takes the form $k \leq k+1$ for some $k \in \bbN$, $k > 0$.   Since no points such that $\mult_p(D) > 1$ can collide without leaving $W_{\mu^{> 1}}$,  the map on $W_{\mu^{> 1}}/\fH_2(\cC_g)$  that extracts the points of $D$ and the multiplicities of $D,D'$ at these points is continuous.  

	In general, if $K \subseteq Y$ is a closed subspace of a fiber bundle over $B$, then the subspace $$\{s \in \Gamma(B,Y), T \in \Sym^l Y ~|~ s(T) \in \Sym^l(K)\}$$ is closed since it is the preimage of $\Sym^n K \subseteq \Sym^n Y$ under the evaluation map.  Applied to $0 \times C \subseteq J^1\Omega^1(-D)$, we obtain that condition (2) is closed, and applied to $T^1 \Omega^1(-D) \subseteq J^1\Omega^1(-D)$, we obtain that condition (1) is closed. \end{proof}

We now let $U \to W_{\mu^{>1}}$ be a classical open subset of a finite \'etale cover of by a scheme.  Pulling back, we have a continuous stratification of $E^{\top}|_U$, given by $Z^{\top}|_U$.   We apply Proposition \ref{fdApprox}, to obtain a chain of semialgebraic vector subbundles $\cF^{(k)} \subseteq E^\top|_U$, containing the subspace of holomorphic sections $\uGamma^{\rm hol}(C,J^1\Omega^1(-D))$, and such that $\cF^{(k)} \cap G$ approximates the homology of any open subset $G \subseteq E^\top|_U$.  We let $Z^{(k)} := Z^{\top}|_U \cap \cF^{(k)}$, be the associated continuous stratification of $\cF^{(k)}$ by $U/\fH_2(\cC_g)$.   (Note that for convenience, we are omitting $U$ from our notation for $Z^{(k)}$).

Again, we take $N = \lfloor 2g/3 \rfloor$ and pass to the Bar construction $\rB(U, \fH_2^{\leq N}, Z^{(k)})$.   In parallel to Theorem \ref{algebraicResolution}, we have the following.

\begin{thm}\label{topologicalResolution}
 	Let $k \in \bbN$.  Let $g \in \bbN$ and $\mu$ be an integer partition of $2g-2$. Let $e^{(k)} :=  2(3g-3 + \#\{i ~|~ \mu(i) > 0\} + \rank_{\bbC} \cF^{(k)} - \deg(\mu^{>1}))$.  Assume that $\lfloor 2 g/3 \rfloor \geq \deg(\mu^{>1}) + 2$.  Then $$H^i_c(\rB( U, \fH_2^{\leq \lfloor 2g/3 \rfloor}, Z^{(k)})) \leftarrow H^i_c(\Disc_{Z^{(k)}}) $$ is connected in codimension $\lfloor \frac{\lfloor2g/3 \rfloor - \deg(\mu^{>1})}{2}\rfloor + 1$ relative to $e^{(k)}$.  
\end{thm}
\begin{proof}
	Again, we will apply Theorem \ref{mainConvergence}.  Let $T$ be an essential multiset with $\cN_T \subseteq \fH_2^{\leq \lfloor 2g/3\rfloor}$.      First we observe for any $x \in \cN_T|_U$ we have  $$Z^{\rm top}_{x} \cap E^{\rm alg}_x = Z^{\rm alg}_x.$$  (Here we write $Z^{\rm alg}$ to denote $E$ from Section \ref{sec:algConvergence}, and $Z^{\rm alg}_x$ is its pullback along $x \to \cN_T|_U \to  W_{\mu^{>1}}/\fH_2$).   Indeed, supposing $x$ maps to $(C, D \leq D') \in W_{\mu^{>1}}/\fH_2(\cC_g)$, the left hand side is the subspace of algebraic sections of $\Gamma(C, \Omega^1(-D))$ such that $s|_{D'-D} = 0$.    Because $T$ is essential $D'-D$ only has multiplicity $1$ or $2$ at every point $p \in C$.  If $p$ has multiplicity $1$, then at $p$ the condition $s|_{D'-D} = 0$ is equivalent to $(j^1s)_0(p) = 0$.  If $p$ has multiplicity $2$ then at $p$ the condition $s|_{D' - D} = 0$ is equivalent to $(j^1s)(p) = 0$.  These are precisely the conditions defining $Z^{\rm top}_x$.  
	
	It follows that $\codim_{\bbC}(Z^{\rm alg}_x  \subseteq E_x^{\rm alg}) \leq   \codim_{\bbC}(Z^{(k)}_x  \subseteq \cF^{(k)}_x)$ at every $x \in \cN_T$, and therefore the codimension of $\codim_\bbR (Z^{(k)}|_{\cN_T|_ U}$ with respect to $\cF^{(k)}|_U $ is greater than or equal to the codimension of $Z^{\rm alg}|_{\cN_T|_U}$ with respect to $E^{\rm alg}|_U$, which is $\geq \alpha(T)$ as in the paragraph preceding Theorem \ref{algebraicResolution}.  
\end{proof}

\begin{remark}
	For the specific choice of $\cF^{(k)}$ used in the proof of Proposition \ref{fdApprox}, it is actually true that $H^i_c(\rB( U, \fH_2^{\leq \lfloor 2g/3 \rfloor}, Z^{(k)})) \leftarrow H^i_c(\Disc_{Z^{(k)}})$ becomes connected in arbitrarily high codimension with respect to $e^{(k)}$ as $k \to \infty$.    We will not use this improvement, because it does not affect the range of Theorem \ref{comparison_theorem}.
\end{remark}

As a consequence of Theorem \ref{topologicalResolution}, we again obtain a description of a spectral sequence converging to the compactly supported cohomology of the  complement of $\cH_{g, \mu}^{\rm top} \cap \cF^{(k)}$ in $\cF^{(k)}$.

 \begin{cor}\label{continuousSpectralSequence}
	 Let $\kappa = \lfloor \frac{\lfloor2g/3 \rfloor - \deg(\mu^{>1})}{2}\rfloor$. Let $e^{(k)} :=  2(3g-3 + \#\{i ~|~ \mu(i) > 0\} + \rank_{\bbC} \cF^{(k)} - \deg(\mu^{>1}))$. There is a spectral sequence that converges to the compactly supported cohomology of the complement of $\cH_{g,\mu}^{\rm top}|_U \cap \cF^{(k)}$  in $\cF^{(k)}$ in degrees $\geq e^{(k)} - \kappa$, with total $\bE_1$ page taking the following  form in degrees $\geq e^{(k)} - \kappa$: $$\bigoplus_{p + q \geq e^{(k)} - \kappa } \bE_1^{p,q} = \bigoplus_{T \text{ essential, }\cN_{T} \subseteq (W_{\mu^{>1}}/^\circ \fH_2^{\leq N})} ~ \bigoplus_{i \geq e^{(k)}- \kappa } H^{i + e^{(k)} - 2 \deg(T)}_c(\cN_T|_U).$$ 
\end{cor}

\subsection{Comparing resolutions} \label{sec:compare} We now prove Theorem \ref{comparison_theorem} using a map between the  spectral sequences associated to the two resolutions constructed in \S \ref{sec:topRes} and \S \ref{sec:algConvergence}.    The argument is directly parallel to \cite[\S 7]{DasTost}.


\begin{proof}[Proof of Theorem \ref{comparison_theorem}] 

First note that for $i \in \bbN$  we have that  $$ i \leq \left\lfloor \frac{\lfloor2g/3 \rfloor - \deg(\mu^{>1})}{2}\right\rfloor  \iff i \leq   m_1(\mu)/2  - 2g/3 + 1,$$ using the identity $\deg(\mu^{>1}) = 2g - 2 - m_1(\mu)$, and elementary algebra.  Thus it suffices to prove that $H_i(\cH_{g,\mu}) \to H_i(\cH_{g,\mu}^{\top})$ is an isomorphism for all $i \leq  \left\lfloor \frac{\lfloor2g/3 \rfloor - \deg(\mu^{>1})}{2}\right \rfloor - 1$.  For notational simplicity, we put $\kappa = \left\lfloor \frac{\lfloor2g/3 \rfloor - \deg(\mu^{>1})}{2}\right\rfloor $.

  By Mayer--Vietoris, we reduce to proving that $H_i(\cH_{g,\mu}|_U) \to H_i(\cH_{g,\mu}^{\top}|_U)$ is an isomorphism  for any $U$ that is a classical semialgebraic open subset of a finite \'etale cover of $W_{\mu^{>1}}$ by a scheme, and for all $i \leq \kappa - 1$.   Fixing such a $U$,  we choose $\cF^{(k)}$ as in Proposition \ref{fdApprox}.  Since $\cH_{g,\mu}^{\top}|_U$ is an open subset of $\uGamma^{\rm top}(C, J^1 \Omega^1(-D))$, we are reduced to proving that the map  $H_i(\cH_{g,\mu}|_U) \to H_i(\cH_{g,\mu}^{\top}|_U \cap \cF^{(k)})$ is an isomorphism for all $i \leq \kappa$.  
  
  The spaces $\cH_{g,\mu}|_U$ and  $\cH_{g,\mu}^{\top}|_U \cap \cF^{(k)}$ are smooth and oriented manifolds, of dimension $e$ and $e^{(k)}$ respectively (where $e, e^{(k)}$ are as defined in Theorems \ref{algebraicResolution} and \ref{topologicalResolution} respectively). Therefore the map on $H_i$ is Poincar\'e dual to the Gysin map  $$\theta_!:H^{e - i}_c(\cH_{g,\mu}|_U) \to H^{e^{(k)} - i}_c(\cH_{g,\mu}^{\top}|_U \cap \cF^{(k)})$$ associated to the orientation classes of each space. We are reduced to proving that this Gysin map is an isomorphism for all $i \leq \kappa$.  
  
  By the long exact sequences in compactly supported cohomology, and their compatibility with Gysin maps, it suffices to prove that $\theta_!:H^{e - i}_c(\Disc_Z|_U) \to H^{e^{(k)} - i}_c(\Disc_{Z^{(k)}})$ is an isomorphism for all $i \leq \kappa$.   The map  $\rB(U, \fH_2^{\leq N}, Z^{\rm alg}) \to \rB(U, \fH_2^{\leq N}, Z^{(k)})$, and compatibility of Gysin map with the compactly supported cohomology spectral sequence associated to a filtration, gives a morphism between the spectral sequences of Corollary \ref{algebraicSectionSpectralSequence} and Corollary \ref{continuousSpectralSequence}, inducing an isomorphism on the $E_1$ page in degrees $\leq \kappa +1$. (See \cite[\S 2]{DasTost} for discussion of compatibility between Gysin maps and the compactly supported cohomology spectral sequence).  Since by Theorems \ref{algebraicResolution} and \ref{topologicalResolution}, these spectral sequences converge to the compactly supported cohomology of $\Disc_Z|_U$ and $\Disc_{Z^{(k)}}$ in degrees $\leq \kappa$,  we are done.  
 \end{proof}

\subsection{Pointed variant}

In order to prove Theorem \ref{stabilizationtheorem} we will use the following variant of Theorem \ref{comparison_theorem} for spaces of framed pointed curves.  

\begin{defn} Let $({\cH_{g, \mu}^1})^{\rm top}$ be the topological stack parameterizing the data of:

	\begin{enumerate} \item $(C,q, t) \in \cM_g^1$  a genus $g$ curve with a point $q$ and a nonzero tangent vector $t \in T_q C$, \\
				\item  $D = \underset{j, \mu(j) > 1} \sum \mu(j) p_j$ a divisor in $\Hilb^{2g-2 - \#\{j, \mu(j) = 1\} }  (C)$, disjoint from $q$\\
				
				\item  and $s \in \Gamma^{\rm top}(C, J^1 (\Omega^1_{C}(-D)))$ a nonvanishing continuous section of the first jet bundle of $\Omega^1_C(-D)$, such that $s_0(p) \neq 0$ for all $p$ in the support of $D$ and $s_0(q) \neq 0$.
\end{enumerate}
\end{defn}

Again there is a comparison map $\cH_{g, \mu}^1 \to ({\cH_{g, \mu}^1})^{\rm top}$, given by taking $(C,q, t, \omega)$ to $(C,q,t, \bbV(\omega)^{>1}, d\omega)$.  (Note that in the introduction we suppressed the tangent vector $t$ from our notation). 

\begin{thm}\label{pointedcomparison}
	$H_i(\cH_{g, \mu}^1) \to H_i(({\cH_{g, \mu}^1})^{\rm top})$ is an isomorphism for all $i \leq \mainbound - 1$.  
\end{thm}
\begin{proof}
	Let $\cH_{g,1,\mu}$ and $\cH_{g,1,\mu}^\top$ denote the spaces obtained by forgetting the data of the framing $t$.   Then the map $\cH_{g, \mu}^1 \to ({\cH_{g, \mu}^1})^{\rm top}$ splits as a product of $\cH_{g, 1,\mu} \to \cH_{g,1,\mu}^\top$ with $\bbC^*$,  using $\omega(q)$.   We describe how to adapt the proof of Theorem \ref{comparison_theorem} to show that $H_i(\cH_{g, 1,\mu}) \to H_i(\cH_{g,1,\mu}^\top)$ is an isomorphism for $i \leq \mainbound - 1$.  
	
	 To compute the compactly supported cohomology of $\cH_{g,1,\mu}$, we use the poset $W_{\mu^{>1} + 1}/\fH_2(\cC_g)$, where $\mu^{>1} +1$ is the integer partition obtained from $\mu^{>1}$ by adding an additional $1$.  A point in $W_{\mu^{>1} + 1}$ is determined by a curve $C$, a divisor $D_2$ with all points of multiplicity $\geq 2$ and a point $q \in C$ disjoint from the support of $D_2$ (yielding a divisor is $D_2 + q$).  We let $E_{\rm point} := \uGamma(C, \Omega^1_C(q)(-D_2 - q)) = \uGamma(C, \Omega^1_C(-D_2))$.   There is a continuous stratification $Z_{\rm point}$ of  $E_{\rm point}$ by $W_{\mu^{>1} + 1}/\fH_2(\cC_g)$,  defined as in \S \ref{sec:algConvergence} so that the fiber above $(C, D_2 + q \leq D') \in W_{\mu^{>1} + 1}/\fH_2(\cC_g)$ is  $\Gamma(C, \Omega^1_C(q)(-D'))$.    Then the discriminant locus $\Disc_{Z_{\rm point}} \subseteq E_{\rm point}$ is the complement of $\cH_{g,1,\mu} \subseteq E_{\rm point}$:   a section $s \in \Gamma(C, \Omega_C^1(-D_2))$ lies in the discriminant locus if and only if it vanishes to order $> 1$ at some point of $C - q - \supp(D)$ or it vanishes to order $1$ at $q$.   As in \S \ref{sec:algConvergence} we apply Theorem \label{mainConvergence} for $N = \lfloor 2g/3 \rfloor$ to show that $H^i_c(\rB(U, \fH_2^{\leq N}, Z_{\rm point})) \leftarrow  H^i_c(\Disc_{Z_{\rm point}})$ is connected in codimension $\lfloor \frac{N - \deg(\mu^>1 + 1)}{2} \rfloor + 1$, so that Theorem \ref{spectralSequence} yields a spectral sequence converging to the compactly supported cohomology of the discriminant locus.  Because $\deg(\mu^{>1} + 1) = \deg(\mu^{>1}) + 1$, the numerics differ by at most one.
	 
	 The remainder of the argument is parallel to the argument in \S \ref{sec:topRes} and \S\ref{sec:compare}.  We construct a finite dimensional approximation to $({\cH_{g, \mu}^1})^{\rm top}$ and a spectral sequence converging to the compactly supported cohomology of the complement of this finite dimensional approximation.   Then the Gysin map between spectral sequences induces an isomorphism on compactly supported cohomology in a range of degrees (shifted by one)
\end{proof}

\section{Stable homology of $\cH_{g,\mu}^{\rm top}$} \label{sec:stablecomputation}

To compute the cohomology of $\cH_{g,\mu}^{\rm top}$ in the stable range, we will apply results of Randall-Williams \cite{RW} on homological stability for spaces of curves with a tangential structure, and Galatius--Madsen--Tillman--Weiss  \cite{GMTW} on the homotopy type of cobordism categories.

\subsection{Homological stability for moduli spaces of surfaces with tangential structure}
	In \cite{RW}, Randall--Williams establishes a general homological stability criterion for moduli spaces of surfaces $\Sigma$ with tangential structure.  We recall his notation here, and describe how it applies to our setting.  Here a  \textbf{(oriented) tangential structure} is a fibration $\theta: X \to B\SO(2)$, and a \textbf{$\theta$-structure} on oriented surface $\Sigma$ is a lift of the  map $\Sigma \to B \SO(2)$ (classifying the tangent bundle of $\Sigma$) to $X$.  
	
	We model $B\SO(2)$ as infinite complex projective space $\bbP^\infty$.  The tangential structure that is relevant to studying strata of $k$-fold differentials is \begin{equation}\label{tangentialstructure} \theta_k : ((O(k) \oplus O(k+1)) - 0) \to \bbP^\infty, \end{equation} because a $\theta_k$ structure on a surface $\Sigma$ corresponds to a nonvanishing section of 
	 $$\Omega_{\Sigma}^{\otimes k} \oplus \Omega_{\Sigma}^{\otimes k+1} \iso J^1(\Omega_{\Sigma}^{\otimes k}).$$ Here $O(k)$ denotes the $k$th tensor power of the dual of the universal sub line bundle $O(-1)$.

		Recall that a bundle morphism from a vector bundle $P \to Y$ to a vector bundle $L \to X$ is a continuous map $P \to L$ covering a map $Y \to X$ that induces a linear isomorphism on each fiber.  (Equivalently a morphism $X \to Y$ together with a choice of isomorphism $P \iso f^* L$).

		 Let $\Sigma$ be an oriented surface with boundary $\del \Sigma$.  A \textbf{boundary condition} is a bundle morphism $\ell_{\del \Sigma}: \bbR \oplus T(\del \Sigma) \to \theta^* O(-1)$.   
		The \textbf{moduli space of $\theta$-surfaces of topological type $\Sigma$} is $$\cM^\theta(\Sigma, \ell|_{\del \Sigma}):= \Bun_{\del}(T \Sigma, \theta^* O(-1); \ell_{\del \Sigma})//{\rm Diff}^{+}_{\del}(\Sigma),$$ the space of bundle morphisms $T \Sigma \to \theta^* O(-1)$ that restrict to $\ell_{\del \Sigma}$ on the boundary,  homotopy quotiented by the topological group of orientation-preserving diffeomorphisms restricting to the identity on a collar neighborhood of the boundary.   (Randall--Williams adopts slightly different, but equivalent, definition in order to include the non-orientable case as well).

	We have the following analog of \cite[Proposition 6.7]{RW}.
		\begin{prop}\label{stabilizationPi0}
			For the tangential structure $\theta_k$ of \eqref{tangentialstructure}, we have that $\cM^{\theta_k}(\Sigma, \ell|_{\del \Sigma})$ is connected for all boundary conditions $\ell_{\del \Sigma}$.  Therefore $\theta_k$ stabilizes on $\pi_0$ starting from genus $0$ and is $1$-trivial in the sense of Randall--Williams \cite{RW}.
		\end{prop}
		\begin{proof} 
			We use obstruction theory to show connectedness of $\Bun_{\del}(T \Sigma, \theta_k^* O(-1); \ell_{\del \Sigma})$.  Indeed, since $\theta_k: (O(1) \oplus O(2))-0 \to \bbP^\infty$ is equivalent to an $S^3$ bundle, any boundary condition  on $\del \Sigma$ admits an extension to to a $\theta_k$-structure on $\Sigma$.  And there is a homotopy between any two such extensions that fixes the structure on the boundary.   Then because $\cM^{\theta_k}(\Sigma, \ell|_{\del \Sigma})$ is a homotopy quotient, it is also path connected.  Finally by definition $\theta$ stabilizes at genus $0$, and so is $2*0 + 1$ trivial by Randall--Williams \cite[Proposition 6.5]{RW} .  
		\end{proof}
		
		By Proposition  \ref{stabilizationPi0}, the numerics are the same as in the case of orientable surfaces with no additional structure, and we may therefore apply  \cite[Theorem 7.1]{RW} and \cite[Theorem 12.1]{RW}   (with $F(g) = \lfloor(2g +1)/3 \rfloor$ and $G(g) = \lfloor 2g /3 \rfloor$) to obtain the following.   
		
		\begin{thm}\label{RandallWilliams}
			Let $\Sigma_{g,b}$ be the genus $g$ surface with $b$ boundary components.  Then
	 \begin{enumerate}[(i)]
			\item Any stabilization map $\alpha_* : H_i(\cM^{\theta_k}(\Sigma_{g, b}, \ell_{\del \Sigma_{g,b}})) \to H_i(\cM^{\theta_k}(\Sigma_{g+1, b-1}, \ell_{\del \Sigma_{g+1,b-1}}))$ obtained by gluing a pair of pants along the legs is an isomorphism for $i \leq  \lfloor(2g +1)/3 \rfloor - 1 $ and a surjection for $i = \lfloor(2g +1)/3 \rfloor$. 
			\item Any stabilization map $\beta_* : H_i(\cM^{\theta_k}(\Sigma_{g, b}, \ell_{\del \Sigma_{g,b}})) \to H_i(\cM^{\theta_k}(\Sigma_{g, b+1}, \ell_{\del \Sigma_{g,b+1}}))$ obtained by gluing a pair of pants on the waist is an isomorphism for $i \leq \lfloor 2g /3 \rfloor - 1 $ and a surjection for $i = \lfloor 2g/3 \rfloor$.
			\item Any stabilization map $\gamma_* : H_i(\cM^{\theta_k}(\Sigma_{g, b+1}, \ell_{\del \Sigma_{g,b+1}})) \to H_i(\cM^{\theta_k}(\Sigma_{g, b}, \ell_{\del \Sigma_{g,b}})$  obtained by gluing in a disk, is an isomorphism for $i \leq \lfloor 2g /3 \rfloor$. For $b > 0$ it is a surjection in all degrees; for $b=0$ it is an surjection for $i \leq \lfloor 2g /3 \rfloor + 1$.
	\end{enumerate}
Finally,  the map $$H_i(\cM^{\theta_k}(\Sigma_{g, b}, \ell_{\del \Sigma_{g,b}})) \to H_i(\Omega^\infty \bM \bT(\theta)_{[\Sigma_{g,b}]})) $$ is an isomorphism for all $i \leq \lfloor 2g /3 \rfloor -1$.   (Here the subscript $[\Sigma_{g,b}]$ denotes the unique component of $\Omega^\infty \bM \bT(\theta_k)$ in the image of the scanning map).  
		\end{thm}

The last statement of the theorem is the most important for us.  To elaborate, we recall that the \textbf{Madsen--Tillmann spectrum}, denoted $\bM \bT(\theta)$, associated to a tangential structure $\theta: X \to \bbP^\infty$ is the Thom spectrum of the virtual line bundle $- (\theta^* O(-1))$,  often denoted as $\Th(-\theta^* O(-1) \to X)$.		
		
\subsection{Relation between  $\cH_{g,\mu}^{ \top}$ and moduli spaces of $\theta_k$-surfaces}
	
	We generalize the definition of $\cH_{g,\mu}^{\top}$ to incorporate higher-order differentials.
\begin{defn} Let $k \in \bbN$ and let $\mu:[n] \to \bbN$ be a partition of $k(2g-2)$.  We let 
$\cH_{g,\mu}^{\otimes k, \top}$ be the topological stack parameterizing the data of 
\begin{enumerate}
\item of an algebraic curve $C \in \cM_g$
\item a divisor $D \subseteq C$ of the form $D = \sum_{i,  \mu(i) \neq 1} \mu(i) p_i$ for $p_i \in C$.
\item a continuous section $s \in \Gamma^{\rm top}(C, J^1 ((\Omega^1)^{\otimes k} (-D) )- 0)$ such that $s_0$ does not vanish at the points of $D$,  where $s_0$ is the $0$th component of the jet.   
\end{enumerate} 
$\cH_{g,\mu}^{\otimes k, \top}$ is a fiber bundle over the base $\cM_{g,n-m_1(\mu)}/\bS_{\mu}$.
\end{defn}
\noindent We will now describe the homotopy type of $\cH_{g,\mu}^{\otimes k, \top}$ in terms of the moduli spaces $\cM^{\theta_k}$.   To do so, we first introduce some auxiliary notation.

  We write $\bC_n$ for the group of $n$th roots of unity (we avoid choosing an identification of $\bC_n$ with $\bbZ/n$, though the reader is free to fix one).

\begin{defn}
	Given an integer partition $\lambda = 1^{m_1} 2^{m_2} \dots $,  we let $\bS \wr \bC_\lambda$ be the group $\prod_{j} (\bS_{m_j} \ltimes \bC_{j}^{\times m_j})$.  
\end{defn}

Note that homotopy classes of map $f:X \to B (\bS_{m_j} \ltimes \bC_{j}^{\times m_j})$ are in bijection with  isomorphism classes of the following data
\begin{itemize} \item a covering space $p: E \to X$ a equipped with free action of $\bC_j $ such that each fiber has exactly $m_j$ orbits.
\end{itemize}

Let $\mu^{+k}$ be the partition $\mu^{+k}(i) := \mu(i) +k$ if $\mu(i) > 0$, and $\mu^{+k}(i) = 0$ otherwise.  We define the \textbf{prong classifying map}: $${\rm prong}:  \cH_{g,\mu}^{\otimes k} \to B( \bS \wr \bC_{\mu^{+k}})$$ in the following way.  Note for $p \in C$ there are canonical isomorphisms $$(\Omega^1)^{\otimes k} (-jp)|_{p} \iso (\Omega^1)^{\otimes k}|_{p}  \otimes (\fm_{p}^{j}/{\fm_p}^{j+1}) \iso (\Omega^1|_p)^{\otimes j + k}.$$ Therefore, for a curve $C$ and a divisor $D$, a section of  $s \in (\Omega^1)^{\otimes k}(-D)$ that does not vanish at $D$ determines a nonzero element  $s(p) \in (\Omega^1|_p)^{\otimes j + k}$ at each multiplicity $j$ point $p \in D$.  Then $\{ t \in \Omega^1|_p ~|~ t^{\otimes j+k} = s(p)\}$ is a torsor for $\bC_{j+k}$, known as the set of \textbf{prongs} of $s$ at $p$.  Taking the union over all multiplicity $j$ points of $D$, we obtain a covering space that classifies a map $B (\bS_{m_j} \ltimes (\bC_{j+k})^{\times m_j})$.  We define the prong classifying map to be the product of these maps over all $j$.  The same procedure, restricted only to zeros of multiplicity $> 1$ yields a map $\prong^{> 1}: \cH_{g,\mu}^{\otimes k} \to B (\bS \wr \bC_{(\mu^{>1})^{+ k}})$  factoring through a map $\prong^{> 1}: \cH_{g,\mu}^{\otimes k, \top} \to B (\bS \wr \bC_{(\mu^{>1})^{+k}})$.

The following proposition expresses the relationship between $\cH_{g,\mu}^{\otimes k}$ and $\cM^{\theta_k}$.

\begin{prop}\label{ProngFiberSequence}
	The homotopy fiber of $$\prong^{> 1}: \cH_{g,\mu}^{\otimes k, \top} \to B (\bS \wr \bC_{(\mu^{>1})^{+k}})$$ is equivalent to $\cM^{\theta_k}(\Sigma_{g,b}, \ell_{\del \Sigma})$, where $b$ is the number of nonzero entries of $\mu^{>1}$ and  $\ell_{\del \Sigma}$ is any choice of boundary condition.  Moreover, the monodromy action of the product of roots of unity $\bC_{(\mu^{>1})^{+k}}$  on the homology of $\cM^{\theta_k}(\Sigma_{g,b}, \ell_{\del \Sigma})$ is trivial.  
\end{prop}
\begin{proof}

First we replace the base of the fibration $\cH_{g,\mu}^{\otimes k, \top} \to \cM_{g,\mu}$ by the equivalent space $\cM_{g,\mu}^{\rm tub}$ parameterizing curves and divisors $(C,D)$ together with a choice of closed neighborhood $N_i$ around each point $p_i$ in the support of $D$ such that  $\cup_{i} N_i$ is a tubular neighborhood of $\{p_i\}_{i = 1}^b$.

Base changing to $\cM_{g,\mu}^{ \rm tub}$ a section is determined by its value on each $N_i$ and on $C - \cup_{i} \overset{\circ} N_i$,  and so there is a fiber square
\[\begin{tikzcd}
{ \cH_{g, \mu}^{\otimes k, \top}} \arrow[d] \arrow[r]                               & {\uGamma^\top(C -  \cup_{i = 1}^n \overset{\circ} N_i , J^1(\Omega^1)- 0)} \arrow[d]               \\
\prod_{i = 1}^{b} \uGamma^\top(N_i,  J^1 (\Omega^1(-D)) -  T^1 (\Omega^1(-D))) \arrow[r] & {\prod_{i = 1}^{b} \uGamma^\top(\del N_{p_i}, J^1 \Omega^1 - 0)}
\end{tikzcd}. \]
	The products in this diagram denote fiber products over the base $\cM_{g,\mu}^{\rm tub}$.  Here we have used that after restricting to $N_i$ there is a canonical identification $\Omega^1(-D) \iso \Omega^1$.   The right vertical evaluation map is a fibration, and therefore this square is a homotopy fiber square.  Furthermore $ \uGamma^\top(N_i,  J^1 (\Omega^1(-D)) -  T^1 (\Omega^1(-D))) \simeq \uGamma^\top(p_i,  J^1 (\Omega^1(-D)) -  T^1 (\Omega^1(-D))) \simeq  \Omega^1(-D)|_{p_i} - 0 $.  
	
	  Base-changing to  the  $\bS_{\mu^{>1}}\wr \bbC^*$ bundle $\tilde{\cM_{g,\mu^{>1}}} \to \cM_{g,\mu^{>1}}$ corresponding to a choice of labeling and a framing of the points $p_i$, this homotopy fiber square  is equivalent to 
	 	
	\[\begin{tikzcd}
{ \tilde \cH_{g, \mu}^{\otimes k, \top}} \arrow[d] \arrow[r]                               & {\tilde{\uGamma^\top(C - D, J^1(\Omega^1)- 0)}} \arrow[d]               \\
 \tilde{\cM_{g,\mu^{>1}}} \times \prod_{i = 1}^{b}  \bbC^*/\bC_{\mu(i)} \arrow[r] & \tilde{\cM_{g,\mu^{>1}}} \times \prod_{i = 1}^{b} \Top(\bbC^*, \bbC^*),
\end{tikzcd} \] after using the framing and the exponential map to trivialize $\Omega^1|_{p_i}$  and to identify $N_i$ with $\bbC$.  The bottom row of this square is pulled back from $\prod_{i = 1}^{b}  \bbC^*/\bC_{\mu(i)}  \to  \prod_{i = 1}^{b} \Top(\bbC^*, \bbC^*)$ along the projection $\tilde{\cM_{g,\mu^{>1}}} \to *$.  Because $\tilde{\cM_{g,\mu^{>1}}}$ is a classifying space for $\Diff^+(\Sigma_{g,b}, \del \Sigma_{g,b})$, the fiber of $\tilde{\uGamma^\top(C - D, J^1(\Omega^1)- 0)} \to \prod_{i = 1}^b \Top(\bbC^*, \bbC^*)$ above a tuple of maps $\ell$ is equivalent to $\cM^{\theta_k}(\Sigma_{g,b}, \ell)$.  Hence the same is true for the map $$\tilde {\cH_{g, \mu}^{\otimes k, \top}}  \to \prod_{i = 1}^{b}  \bbC^*/\bC_{\mu(i)}.$$ Taking the homotopy quotient by of this map the deck group $\bS_\mu \wr \bbC^*$ gives the prong classifying map, so the homotopy fiber of the prong classifying map is also $\cM^{\theta_k}(\Sigma_{g,b}, \ell)$.  The monodromy action by roots of unity is trivial because $\prod_{i = 1}^b \bC_i \subseteq \bS_\mu \wr \bbC^*$ acts trivially on the homology of $\cM^{\theta_k}(\Sigma_{g,b}, \ell)$.
\end{proof}		

To apply Proposition \ref{ProngFiberSequence}, we must understand the homology of $\cM^{\theta_k}(\Sigma_{g,b}, \ell_{\del \Sigma})$ in the stable range.  By Theorem \ref{RandallWilliams}, it suffices to analyze the homology of the infinite loop space of the Madsen--Tillman spectrum. 

\begin{prop}\label{ThomComputation}
	The rational homology of every component of  $\Omega^\infty \bM \bT(\theta_k)$ vanishes in all degrees $> 0$ (it is rationally contractible).  Localized away from $2k(k+1)$,  $\Omega^\infty \bM \bT(\theta_k)$ splits as a product $\Omega^\infty S^{-2} \times \Omega^\infty S^0$.
\end{prop}
\begin{proof}
	We compute the integral homology of $\bM \bT(\theta_k)$ using the Thom isomorphism and the Leray--Gysin sequence for the $S^3$ bundle associated to $(O(k) \oplus O(k+1)) - 0 \to \bbP^\infty$.  Since the Euler class of $O(k) \oplus O(k+1)$ is $-(k)(k+1) x^2$ where $x := c_1(O(-1))$, we have that the cohomology ring of $(O(k) \oplus O(k+1)) - 0$ is $\bbZ[x]/(-(k)(k+1) x^2)$.  By universal coefficients and the Thom isomorphism, we obtain that $H_{-2}(\bM \bT(\theta_k)) = \bbZ$, $H_0(\bM \bT(\theta_k)) = \bbZ$ and $H_{2i + 1}(\bM \bT(\theta_k)) = \bbZ/(k(k+1))$ for all $i \geq 0$, and all other homology groups vanish.
	
	  Therefore $\bM \bT(\theta_k)[\frac{1}{2k(k+1)}]$ only has homology in degrees $-2$ and $0$, and fits into a fiber sequence $S^{-2}[\frac{1}{2k(k+1)}] \to \bM \bT(\theta_k)[\frac{1}{2k(k+1)}] \to S^{0}[\frac{1}{2k(k+1)}]$.  Because we have inverted $2$,  the connecting map in $$\left[S^{0}[\frac{1}{2k(k+1)}], S^{-1}[\frac{1}{2k(k+1)}]\right] = \pi_{1}^{\rm st}(S^0)[\frac{1}{2k(k+1)}] = 0 $$ is null-homotopic and so $\bM \bT(\theta_k)[\frac{1}{2k(k+1)}]$ splits as a wedge sum of two localized spheres.   Since $\Omega^\infty$ takes sums to products and preserves localizations of desuspensions of connective spectra, we obtain the result.  
 \end{proof}

\subsection{Conclusions}  We will now prove the results stated in the introduction.

\begin{proof}[Proof of Theorem \ref{mainComputation} and \ref{modPExample}]
	By Theorem \ref{comparison_theorem}, to calculate (co)homology in the stable range, it suffices to compute the stable cohomology of $\cH_{g, \mu}^{\top}$.  To do so, we use Proposition \ref{ThomComputation} (which describes the stable (co)homology of $\cM^{\theta_1}(\Sigma_{g, b}, \ell_{\del \Sigma_{g,b}})$, by Theorem \ref{RandallWilliams})  together with the Leray--Serre spectral sequence for the fibration of Proposition \ref{ProngFiberSequence}.
	
	In the case of rational (co)homology, Theorem \ref{mainComputation}, we have that the higher degree rational homology of the base is trivial, and that of the fiber is trivial in the stable range.  So rational (co)homology vanishing in higher degrees follows.   In the case of Theorem \ref{modPExample}, our assumption on $\mu$ also guarantees that the $\bbF_p$ homology of the base is trivial (as a classifying space of a finite group of order relatively prime to p) and that the monodromy action is trivial.  So the homology of the total space agrees with the homology of the fibers.
\end{proof}

\begin{proof}[Proof of Theorem \ref{stabilizationtheorem}] 
 By Theorem \ref{pointedcomparison},  the map $\cH_{g,\mu}^1 \to (\cH_{g,\mu}^1)^{\rm top}$ induces an isomorphism on $H_i$ for $i \leq \mainbound - 1$.  The homotopy class of stabilization map $\cH_{g,\mu}^1 \to \cH_{g,\mu +1 +1}^1$ constructed in the introduction is compatible with a topological stabilization map $(\cH_{g,\mu}^1)^{\rm top} \to (\cH_{g,\mu+1+1}^1)^{\rm top}$, given by gluing and smoothing surfaces.
 
We have a fibration $ (\cH_{g,\mu}^1)^{\rm top} \to (\cH_{g,\mu})^{\rm top}$, with fiber the total tangent bundle of $C - \supp(D)$, so that the fiber sequence of Proposition \ref{ProngFiberSequence} yields a fiber sequence $$ \cM^{\theta_1}(\Sigma_{g,b+1}, \ell_{\del \Sigma}) \to (\cH_{g,\mu}^1)^{\top} \to B (\bS \wr \bC_{(\mu^{>1})}).$$   Then the topological stablization map corresponds to a map of fibers $\cM^{\theta_1}(\Sigma_{g,b+1}, \ell_{\del \Sigma}) \to \cM^{\theta_1}(\Sigma_{g+1,b+1}, \ell_{\del \Sigma})$, which is a composition of stabilization maps of falling into cases (ii) and (i) of Theorem \ref{RandallWilliams}.  Applying this theorem, we obtain that the topological stabilization map induces an isomorphism on $H_i$ for $i \leq \lfloor 2g/3 \rfloor -1$.  Therefore $H_i(\cH_{g,\mu}^1) \to H_i(\cH_{g,\mu+1+1}^1)$ is an isomorphism for $i \leq \min (\lfloor 2g/3 \rfloor -1, \mainbound - 1) = \mainbound - 1.$. \end{proof}

To compute the Picard group using cohomology, we adapt an argument from Landesman--Levy. 
\begin{prop}
	Let $U$ be a smooth finite type Deligne-Mumford stack, admitting a smooth proper compactification $X$ with closed complement $Z$ of codimension $\geq 1$.  If $H^1(U, \bbQ) = H^2(U, \bbQ) = 0$ then the natural map $\Pic(U)\to H^2(U, \bbZ)$ is an isomorphism. \end{prop}
\begin{proof}

We use the morphism of Gysin sequences with exact rows, given by the cycle class map.   
\[\begin{tikzcd}
\mathbb Z^{\#{\rm irred}(Z)} \arrow[d, "c_Z"] \arrow[r] & {\rm Pic}(X) \arrow[r] \arrow[d,"c_X"] & {\rm Pic}(U) \arrow[r] \arrow[d,"c_U"] & 0 \arrow[d]          \\
H_{2n - 2}^{\rm BM}(Z) \arrow[r]                 & H^2(X) \arrow[r]                 & H^2(U) \arrow[r]                 & H^{\rm BM}_{2n-3}(Z)
\end{tikzcd}	\]	
See the discussion after equation (7.2) of  Landesman--Levy \cite{LandLev} for the upper row.  The lower row follows from the identification  $H^*(Z,i^! \bbZ) = H^*(Z, i^! \omega_X[-2n]) = H^*(Z, \omega_Z [-2n]) = H_*^{\rm BM}(Z)[2n]$.

The group $H^{\rm BM}_{2n-3}(Z)$ injects into $H^{\rm BM}_{2n-3}(Z- Z^{\rm sing}) \iso H^1(Z- Z^{\rm sing})$ and is therefore torsion-free.   So the hypothesis that $H^2(U)$ is torsion implies that we may replace  $H^{\rm BM}_{2n-3}(Z)$ by zero and remain exact.   Furthermore by \cite[Proposition 7.2.2]{LandLev} $c_X$ is an injection, and by the exponential sequence the quotient is torsion-free.  Since $c_Z$ is surjective, it follows by a diagram-chase (or by considering the spectral sequences associated to the displayed bicomplex) that there is a short exact sequence  $$0 \to \Pic(U) \to H^2(U,\bbZ) \to \coker(c_X) \to 0.$$ But because $H^2(U, \bbZ)$ is torsion and $\coker(c_X)$ is torsion free, we have $\coker(c_X) = 0$, concluding the proof.	
\end{proof}

\printbibliography

\end{document}